\documentclass[]{interact}
\usepackage{bbm}
\usepackage{xcolor}

\usepackage{graphicx}

\allowdisplaybreaks
\newtheorem{theorem}{Theorem}[section]
\newtheorem{lemma}[theorem]{Lemma}
\newtheorem{corollary}[theorem]{Corollary}
\newtheorem{proposition}[theorem]{Proposition}

\theoremstyle{definition}
\newtheorem{definition}[theorem]{Definition}
\newtheorem{example}[theorem]{Example}

\newtheorem{remark}[theorem]{Remark}
\numberwithin{equation}{section}

\newcommand{\R}{\mathbb R}
\newcommand{\Z}{\mathbb Z}
\newcommand{\E}{\mathbf {E}}
\newcommand{\Var}{\mathbf{Var}}
\newcommand{\pr}{\mathbb {P}}
\newcommand{\N}{\mathbb {N}}
\newcommand{\Cov}{\mathbf{Cov}}
\newcommand{\Corr}{\mathbf{Corr}}

\newcommand{\vf}{\varphi}
\newcommand{\rank}{\mathop{\mbox{\rm rank}}}

\begin{document}

%\title[LT for excursion sets of GRF with long memory]{Limit theorems for excursion sets of subordinated Gaussian random fields with long memory}
\title{Limit theorems for excursion sets of subordinated Gaussian random fields with long-range dependence}
\author{
\name{Vitalii Makogin\textsuperscript{a}\thanks{CONTACT V. Makogin. Email: vitalii.makogin@uni-ulm.de, Evgeny Spodarev. Email: evgeny.spodarev@uni-ulm.de} and Evgeny Spodarev\textsuperscript{a}}
\affil{\textsuperscript{a}Institut f\"{u}r Stochastik, Universit\"{a}t Ulm,  D-89069 Ulm, Germany. }
}

\maketitle
\begin{amscode}60G60; 60H05; 60D05; 60G10\end{amscode}


\begin{abstract}
This paper considers the asymptotic behaviour of volumes of excursion sets of subordinated Gaussian random fields with (possibly) infinite variance. Actually, we consider  integral functionals of such fields and obtain their limiting distribution using the Hermite expansion of the integrand. We consider the general non-stationary Gaussian random fields, including stationary and anisotropic special cases. The limiting random variables in our limit theorems have the form of multiple Wiener-It\^{o} integrals. We illustrate most results with corresponding examples.
\end{abstract}

\section{Introduction}

%Excursions of random fields over level $u \in\R$ are random sets of all points $t\in \R^d$ where $X(t)$  exceeds level $u.$  
For a real-valued measurable random field $\{X(t),t\in \R^d\},$ the volume of excursion set $A_u(X,W)=\{t\in W: X(t)\geq u\}$ in observation window {$W\subset  \R^d$} is given by  $$\nu_d(A_u(X,W))=\int_{W} \mathbbm{1}\{X(t)\geq u\}\nu_d(dt).$$ 
Here and further in this paper, {$u$ is a fixed constant, $W$ is a compact subset of $\R^d$},  $\nu_d(\cdot)$ denotes the Lebesgue measure in $\R^d.$ 
Volumes and other geometric characteristics of excursions of random fields are widely used for data analysis purposes in physics and cosmology (see e.g. \cite{Marinucci}),  medicine  \cite{Adler, Taylor}, materials science  \cite{Roubin,Torquato}. 

The volumes of excursion sets $\{A_u(X,W_n),n\geq 1\}$ in observation windows $W_n,$ $n\geq 1,$ form a sequence of random variables. We expect the existence of the limit in distribution
\begin{equation}
\label{lim1} 
\lim_{n\to \infty}\frac{\nu_d(A_u(X,W_n)) - a_n}{b_n}
\end{equation}
for some number sequences $a_n$, $b_n>0, n\in \N,$ as  observation windows $W_n$ grow in van Hove sense (see, e.g., \cite[Chapter 3]{BulSh}), i.e., $\nu_d(W_n)\to \infty, n\to \infty$ and
$\lim_{n\to \infty}\nu^{-1}_d(W_n)\nu_d(\partial W_n \oplus B_r(0))=0,$ $r>0,$ {where $\partial W_n$ denotes the boundary of $W_n,$ the Minkowski sum of two sets $A, B \subset \R^d$  is  $A \oplus B = \{x + y: x \in A,y \in B\},$ and $B_r(0)=\{x \in \R^d, \|x\|<r\}.$} 

During past decades, a significant contribution was made to find the limiting distribution in \eqref{lim1} for isotropic and/or stationary random fields, see \cite{Leonenko88} and the books \cite{Leonenko86, Leonenko99}.  The result of Bulinski et. al. \cite{Bulinski2012} states that limiting distribution in \eqref{lim1} is Gaussian if $X$ is a Gaussian centered stationary random field with continuous covariance function $C(t)=\E[X(0)X(t)]$ such that $|C(t)|=O(\|t\|^{-\alpha})$ for some $\alpha>d$ as $\|t\|\to \infty.$ 
In case $0<\alpha<d,$ the field $X$ is long-range dependent and such result can not be used. 
\begin{definition}
A square-integrable stationary centered random field $\{X(t),t\in \R^d\}$ with covariance function $C(t)=\E[X(0)X(t)],t\in \R^d$ is called \textit{long-range} dependent (or with long memory) if \begin{equation}
\label{sigm3}
\int_{\R^d} |C(t)|\, dt = + \infty
\end{equation}
and \textit{weakly} dependent if this integral is finite.
\end{definition}
%Roughly speaking, a stochastic process (for $d=1$) is long-range-dependent when the prediction of its next state depends on the whole of its past. The divergence of the above integral means that covariance $C$ decays slowly and the values of the process $X$ are strongly dependent over long periods of time. 

%The random functions with long memory and related subjects have been brought to the attention of a wider mathematical audience (and in particular probabilists and statisticians) by the pioneering work of Mandelbrot and his co-workers (e.g. \cite{Ma1966,Ma1968}). 

Real data bring evidence of long memory property in many fields of modern science. For example,  the long memory properties of the final energy demand in Portugal are detected in \cite{belbute}.  An overview of the  state of the art in the theoretical findings for long range dependent stochastic processes can be found, for instance, in \cite{beran} and \cite{sam2016}.

The random fields used in cosmology (potential, temperature, velocity, density of matter, etc.) are mostly Gaussian or derived from Gaussian random fields as their local transformation (Rayleigh, Maxwell, Lognormal and Rectangular processes), see e.g. \cite{coles, coles1991}. 
Frequently used transforms are $f(x)=x+\beta x^3, x\in \R, \beta>0$ (cubic model) and
$f(x)=x+\alpha (x^2-1), x\in \R, \alpha>0$ (quadratic model), cf. \cite{shandarin}. 

For example, in \cite{verde} authors consider a model in which the gravitational potential $\Phi$ is a linear combination of a Gaussian random field $\phi$ and the square of the same random field, $\Phi=\phi+\alpha_\Phi(\phi^2-\E\phi^2),$ where $\alpha_\Phi>0.$ Lognormal random field models with $X(t)=\exp(Y(t)),$ where $Y$ is a Gaussian random field,  are of interest in radar and image processing, see  e.g. \cite{frankot}.  For further physical literature on Gaussian subordinated fields we refer to \cite{bartolo2,bartolo1}.

The theory of random fields with {long-range dependence} is not so developed as for stochastic processes. The first studies on this topic can be found in \cite{Leonenko86,Kelbert,Leonenko99,Leonenko13}).  They prove  limit theorems for functionals of the form
\begin{equation}
\label{LeonLimit}
Z_n=\int_{W_n}G(X(s))ds \mbox{ as }n\to+\infty,
\end{equation}
where $W_n$ are some growing sets, {$G:\R\to\R$ is a measurable function}, and $X$ is an isotropic stationary Gaussian random field with covariance function $C(s,t),$ that depends only on the distance $\|s-t\|.$ This means that probability law of random field $X$ is invariant with respect to rigid motions. {We consider further the case of $G=\mathbbm{1}\{f(\cdot)\geq u\}.$}

Due to results in \cite{Leonenko86}, if the Hermit rank of function $G$ is greater {or equal} $2$ then the limiting distribution in \eqref{LeonLimit} is non-Gaussian. Much earlier, the non-Gaussian limit was found in \cite{Rosenblatt} due to non-summable correlations and non-linearity of function $G$.
This paper led to further developments in 70s and 80s (see e.g. \cite{major,Taqqu}).

Consider the Hilbert space $L^2(\R,\vf),$ with weight $\varphi(x)=\frac{1}{\sqrt{2\pi}}e^{-x^2/2},x\in \R.$  Hermite polynomials $\{H_k\}_{k\geq 0},$ given by 
$$H_k(x)=(-1)^k e^{x^2/2}\frac{d^k}{d x^k}e^{-x^2/2},k\geq 0,
$$
form a complete orthogonal system in $L^2(\R,\vf)$ (see e.g. \cite[Chapter~22]{abram},  \cite[Chapter~2]{doman}, \cite{rozanov}), that is,
$\langle H_k,H_l\rangle_{\vf}:=\int_{\R}H_k(x)H_l(x)\vf(x)dx=\delta_{kl}k!,k \geq 0.
$
The first few polynomials are 
$H_0(x)=1,\, H_1(x)=x, \, H_2(x)=x^2-1, \ldots \, x\in \R.$

\begin{definition}
For a function $G\in L^2(\R,\vf)$ its Hermite rank is 
$$\rank G=\min\{k\in \N|\langle G,H_k\rangle_{\vf}\neq 0\}.$$
\end{definition}

In the last few years, the active research of anisotropic linear random fields with long-range dependence started with papers \cite{Lavancier2006,Lavancier2007} and \cite{Kaj}.  The papers \cite{Pilip17,surg,Surg19-2,Surg19-1} introduced the notions of scaling transition and distributional long-range dependence for stationary linear random fields on $\Z^d.$ %whose normalized partial sums on rectangles with sides growing at rates $O(n)$ and $O(n^\gamma)$ tend to an operator scaling random field $V_\gamma$ on $\R^2,$ for any $\gamma>0.$ The existence of scaling transition together with anisotropic and isotropic distributional long-range dependence properties is demonstrated for a class of $\alpha$-stable aggregated nearest-neighbor autoregressive random fields on $\Z^2.$

{
Another prominent modern technique for  proving limit theorems is based on the Malliavin-Stein approach. Its advantage is a possibility to obtain convergence rates if the limiting distribution is Gaussian and there exist moments of order at least $4+\delta$ ($\delta>0$) of our random field.
Mostly, the central limiting behaviour is proved under weak dependence, see, for example, \cite{pham} for the case of sojourn times of Gaussian fields and \cite{muller} for the case of  Lipschitz–Killing curvatures of Gaussian excursions. To our knowledge, the long- range dependence is covered by several special cases, see e.g. \cite{nurdin}.}

In this paper we extend the above lines of research to  non-stationary random fields. They arise naturally either as weighted/transformed stationary fields or as a result of filtering,  {see the study in  \cite{Alodat20}, \cite{Anh13}, and \cite{Bai} on the limiting behaviour in \eqref{LeonLimit} for filtered random fields.}   Particularly, non-stationary filtered random fields are used in astronomy, when spatial structure is studied by using both the  time and wavelength dimensions and the method of Doppler tomography, see \cite{Vaughan}. For example, a filtered random field model is used for the line emission by  $y(t)=\int_\R \psi(t-\tau)x(\tau)d\tau, t\in \R$ where $x$ is the driving continuum and $\psi$ is the response function.  In theory of fluid flows, the filtered velocity field is given by
$U'(x,t)=\int_{\R^3} G(x,y)U(y,t)\nu_3(dy), x\in \R^3, t>0$ where $U$ is a velocity field and $G$ is a filter. Filters of the form $G(x,y)=\tilde{G}(x-y)$ are called homogeneous. The commonly used in large eddy simulations are Gaussian, Tophat and Sharp Fourier cutoff filters. The inhomogeneous filters, which produce the non-stationary random fields,  reflect local changes in the flow scale. The coordinate-wise product filters of the form $G(x,y)=G_1(x_1,y_1)G_2(x_2,y_2)G_3(x_3,y_3),x=(x_1,x_2,x_3)\in \R^3,y=(y_1,y_2,y_3)\in \R^3$ are also commonly used. Some of $G_1,G_2, G_3$ can be homogeneous. For example, in a channel flow, the stream-wise and span-wise directions are homogeneous and the wall-normal direction is not. For further details, we refer to \cite{bernard}. 

We start the paper with the central limit case in \eqref{lim1} for very general non-stationary random fields. 
More precisely, for a subordinated Gaussian random field $\{X(t)=f(Y(t)),t\in \R^d\},$ where $f:\R\to\R$ is a transformation function and $Y$ is a centered Gaussian random field with $\rho(t,s)=\Corr (Y(t),Y(s)),t,s,\in \R^d,$ we obtain the convergence to $N(0,1)$ in \eqref{lim1} if $\rank \mathbbm{1}\{f(\cdot)\geq u\} =1$ and 
\begin{equation}
\label{eqi1}
\lim_{n\to\infty}\frac{\int_{W_n}\int_{W_n}\rho^2(t,s)\nu_d(dt) \nu_d(ds)}{\int_{W_n}\int_{W_n}\rho(t,s)\nu_d(dt) \nu_d(ds)}= 0.
\end{equation} 

{For the case of  $\rank \mathbbm{1}\{f(\cdot)\geq u\} =1$,} we simplify condition \eqref{eqi1} for non-isotropic stationary covariance functions. We show that it is true only for long-range dependent random fields. We also pay special attention to the spatio-temporal case and show that if {long-range dependence property} is carried by time variable only, then it can be enough to meet condition \eqref{eqi1}.

We also prove limit theorems for the cases $\rank \mathbbm{1}\{f(\cdot)\geq u\} \geq 2.$
To do so, we extend the problem to the limiting behaviour of general integral functionals of Gaussian random fields. The main technique here is the spectral theory and spectral representation of (non-)stationary random fields. The conditions ensuring our limit theorems are formulated via the asymptotic behaviour of spectral densities. 

If the Hermite rank of transformation function $f$ is greater {or equal} 2,  the limiting random variables in \eqref{lim1} have the form of multiple Wiener-It\^{o} integrals (first introduced by It\^{o} in \cite{ito}). {See, for example, \cite{nualart} for their construction and properties.} % defined as 
%\begin{equation*}
%    \mathcal{I}_m(f)=\int'_{\R^{dm}}f(x_1,\ldots,x_m)B(d x_1)\cdots B(d x_m),
%\end{equation*}
%where $\int'$ is an integral excluding diagonals  $x_i = x_j,$ $i\neq j,$ $i, j = 1,\ldots,m,$  $f\in L^2(\R^{dm})$ and $B$ is a real-valued Gaussian random measure on $\mathcal{B}(\R^d)$ satisfying $\E B(A)=0$ and $\E B(A_1)B(A_2)=\nu_d(A_1\cap A_2)$ for any bounded Borel sets $A,A_1,A_2.$

%We use later the relationship between Hermite polynomials and
%complex-valued multiple stochastic integrals defined as follows.  Let $M_1$ and $M_2$ be two independent real valued Gaussian measures on $\mathcal{B}(\R^d).$ Define 
%a complex-valued Gaussian random measure $M$ by  $M(A)=\frac{1}{\sqrt{2}}(M_1(A)+i M_2(A)).$ In particular, for a set $A,$ we have
%$\E M(A)=0$ $\E |M(A)|^2=\nu_d(A).$ Take any  symmetric function $g : \R^{dm}\to \mathbb{C}$, $g(x) = \overline{g(-x)}, x\in \R^{dm},$ which is  invariant under permutation
%of its indices. Similarly to real valued case, one can
%define 
%\begin{equation*}
%    \mathcal{I}_m(g)=\int'_{\R^{dm}}g(y_1,\ldots,y_m) M(d y_1)\cdots  M(d y_m),
%\end{equation*} where
%the integration disregards hyperplanes $|y_i| = |y_j|,$ $y_i\neq y_j.$

{
To summarise, our paper contributes to the theory of limit theorems of random fields as follows:
\begin{itemize}
    \item The paper proves a Gaussian limit for volumes of excursion sets of long-range dependent random fields with the normalization different from CLT (given in \cite{Bulinski2012} under the short-range dependence).
    \item The limit theorems for integral functionals under long-range dependence  based on the Hermite expansion technique (e.g. \cite{leon14}) are extended to the non-stationary Gaussian random fields. Moreover, integration domains now grow in the van Hove sense, which allows for more flexibility in their geometry.  
    \item The subordination allows to get the limit theorems for the class of random fields with infinite variance.  
    \item Numerous examples show that our conditions are relatively easy to check.
\end{itemize}
}

The paper is organized as follows. In Section 2, we consider the central limit theorem in \eqref{lim1} for both non-stationary (Section 2.1) and stationary Gaussian random fields (Section 2.2). Spatio-temporal random fields are covered by Section 2.3. In Section 3, we present the results on non-Gaussian limiting behaviour of integral functionals of non-stationary (Section 3.1) and stationary (Section 3.2) Gaussian random fields.  To illustrate our results, we provide examples of covariance functions and spectral densities matching our theory. Moreover, in Section 4, we consider excursion sets of random fields with random volatility (Section 4.1) and fractional Gaussian noise (Section 4.2) in more detail.

%Reveillac et al. have investigated in \cite{reveillac} the asymptotic behaviour of Hermite polynomials $H_m$ of increments of a two-parametric fractional Brownian sheet $B_{(H_1,H_2)}.$ The authors found that the limiting distribution is Gaussian if $H_1\leq 1 − \frac{1}{2 m}$ or $H_2\leq 1 − \frac{1}{2 m}$

\section{Central limit theorems}
\subsection{Non-stationary random fields}
In this section, we prove the central limit theorem for excursion sets of subordinated Gaussian random fields with {long-range dependence.}
%bounded variance $\sigma^2(t)=\E [Y(t)]^2,$ $ 0<\sigma_1^2\leq \sigma^2(t)\leq \sigma_2^2<\infty$ and
Moreover, we apply it further in Section 4 to the case fractional Brownian motion, fractional Gaussian noise, and random fields with random volatility.

We formulate the following theorem for non-stationary random fields with non-constant variance. This is motivated by seasonal models, in particular econometric time series, e.g \cite{Hannan63, Hannan2001}.  

\begin{theorem}
\label{mainthm}
Let $\{Y(t),t\in \R^d\}$ be a real valued measurable centered Gaussian random field with  correlation function $\Corr(Y(t),Y(s))=\rho(t,s),$ $t,s\in \R^d.$ Let $\{X(t)=f(Y(t)),t\in \R^d\}$ be the corresponding subordinated field, where $f:\R\to\R$ is a Borel-measurable function. Let $(W_n)_{n\in \N}$ be a van Hove sequence of observation windows. For $u\in \R,$ let there exit a subset $U\subset W_n, n\geq 1$ such that $\nu_d(U)>0$ and
\begin{equation}
\label{eq111}
a_1(t)=\langle H_1, \mathbbm{1}\{f(\sigma(t)\cdot)\geq u\}\rangle_{\vf}\neq 0,\,t\in U,
\end{equation}
where $\sigma^2(t)=\E Y(t)^2,t\in \R^d.$  If
\begin{equation}
\label{eq11}
\lim_{n\to\infty}\frac{\int_{W_n}\int_{W_n}\rho^2(t,s)\nu_d(dt) \nu_d(ds)}{\int_{W_n}\int_{W_n}a_1(t)a_1(s)\rho(t,s)\nu_d(dt) \nu_d(ds)}= 0,
\end{equation} 
 then 
\begin{equation}
\label{eq12}
\frac{\int_{W_n}\mathbbm{1}\{X(t)\geq u\}\nu_d(dt)-\int_{W_n}\pr(X(t)\geq u)\nu_d(dt)}{  \left(\int_{W_n}\int_{W_n}a_1(t)a_1(s)\rho(t,s)\nu_d(dt) \nu_d(ds)\right)^{1/2}}\overset{d}{\longrightarrow} N(0,1)
\end{equation}
as $n\to\infty.$
\end{theorem}
\begin{proof} 
Consider the function $F_u(x,t):=\mathbbm{1}\{f(\sigma(t)x)\geq u\},\, x\in\R.$ It is clear that $F_u(\cdot,t)\in L^2(\R,\vf).$
%Indeed, $\int_{\R}F_u^2(x,t)\vf(x)dx\leq \int_{\R}\vf(x)dx =1.$
So, the function $F_u$ can be represented as 
\begin{equation}
\label{eq13}
F_u(x,t)=\sum_{k=0}^{\infty}a_k(t)\frac{H_k(x)}{\sqrt{k!}},~x\in \R, \text{ where } 
a_k(t)=\frac{\langle F_u(\cdot,t),H_k\rangle_{\vf}}{\sqrt{k!}},~k\in \N_0.
\end{equation}
In particular, 
%$$a_0=\int_{u}^{+\infty}\vf(x)dx=1-\Phi(u)=\Psi(u).$$
$a_0(t)=\int_{\R}\mathbbm{1}\{f(\sigma(t)x)\geq u\}\vf(x)dx=\pr(X(t)\geq u)$ {and
%Due to the property $\vf^{(k)}(x)=(-1)^k H_k(x)\vf(x),k\geq 0,$ we have 
%\begin{equation}
%\label{eq15}
%a_k=\int_{u}^{+\infty}\frac{(-1)^k}{k!}\vf^{(k)}dx=\frac{(-1)^k}{k!}\vf^{(k-1)}(u), k\in \N,
%a_k(t)=\int_{\R}\mathbbm{1}\{f(\sigma(t)x)\geq u\}\frac{(-1)^k}{\sqrt{k!}}\vf^{(k)}(x)dx, k\in \N,
%\end{equation}
%so that   \begin{align*}
$a_1(t)%&=-\int_{\R}\mathbbm{1}\{f(\sigma(t)x)\geq u\}\vf'(x)dx\\
%&
=\int_{\R}\mathbbm{1}\{f(\sigma(t)x)\geq u\}x\vf(x)dx=\E[ Y(t)\mathbbm{1}\{f(Y(t))\geq u\}]/{\sigma(t)}.$}
%\end{align*}
Let $\tilde{Y}(t)=Y(t)/\sigma(t),$ $t\in \R^d.$ {Note that $\tilde{Y}(t)\sim N(0,1)$ and  
$\int_{W_n}\sum_{k=0}^{\infty}a^2_k(t) \|H_k(\tilde{Y}(t))\|^2_{L^2(\R,\vf)}/k!\nu_d(dt)= \int_{W_n}\|F_u(\tilde{Y}(t),t)\|^2_{L^2(\R,\vf)}\nu_d(dt) \leq \nu_d(W_n).$ Then by Fubini's theorem we have the expansion in $L^2(\R,\vf)$}
\begin{align}
&\int_{W_n}\mathbbm{1}\{X(t)\geq u\}\nu_d(dt)=\nonumber\int_{W_n}F_u(\tilde{Y}(t),t)\nu_d(dt)=\int_{W_n} \sum_{k=0}^{\infty}a_k(t) \frac{H_k(\tilde{Y}(t))}{\sqrt{k!}}\nu_d(dt)\\
&\nonumber =\sum_{n=0}^{\infty}\int_{W_n}a_k(t) \frac{H_k(\tilde{Y}(t))}{\sqrt{k!}}\nu_d(dt)=\int_{W_n}a_0(t)H_0(\tilde{Y}(t))\nu_d(dt)\\
&\label{repr}+ \int_{W_n}a_1(t)H_1(\tilde{Y}(t))\nu_d(dt)+\sum_{k=2}^{\infty}\int_{W_n}a_k(t) \frac{H_k(\tilde{Y}(t))}{\sqrt{k!}}\nu_d(dt)\\
&\nonumber=\int_{W_n}\pr(X(t)\geq u)\nu_d(dt)+ \int_{W_n}a_1(t)\tilde{Y}(t)\nu_d(dt)+ \sum_{k=2}^{\infty}\int_{W_n}a_k(t) \frac{H_k(\tilde{Y}(t))}{\sqrt{k!}}\nu_d(dt).
\end{align}
Denote
\begin{align*}
%    Y_n:&=\int_{W_n}h_u(X(t))dt-\Psi(u)\nu_d(W_n)\\
    Y_n:&=\int_{W_n}F_u(\tilde{Y}(t),t)\nu_d(dt)-\int_{W_n}\pr(X(t)\geq u)\nu_d(dt),\\
    Z_n:&=\int_{W_n}a_1(t)\tilde{Y}(t)\nu_d(dt), \quad
    A_n:=\sum_{k=2}^{\infty}\int_{W_n}a_k(t) \frac{H_k(\tilde{Y}(t))}{\sqrt{k!}}\nu_d(dt).
\end{align*}
From expansion \eqref{repr} we have $Y_n=Z_n+A_n.$ Random variables  $\{Z_n\}_{n\in \N}$ are Gaussian. So, 
$(\Var Z_n)^{-1/2}(Z_n-\E Z_n)\sim N(0,1).$ Moreover, we prove that $\frac{\Var A_n}{\Var Z_n}\to 0, n\to \infty.$

Since Hermite polynomials form an orthonormal system in $L^2(\R,\varphi),$  we have (cf. \cite[Lemma 10.2]{Roz})
\begin{equation}
\label{eq161}
\E[ H_k(\tilde{Y}(t))H_m(\tilde{Y}(s)) ]=\delta_{km}k! \rho^k(t,s),\quad t,s\in \R^d.
\end{equation}
Moreover, we get 
$\E H_k(\tilde{Y}(t))=\int_\R H_k(x)H_0(x)\vf(x)dx=\delta_{k0}=0,k\in \N.$
%Denote
%\begin{equation}
%\label{eq16}
%\sigma^2_{n,k}:=\Var \left(\int_{W_n}a_k(t) H_k(\tilde{Y}(t))\nu_d(dt)\right).
%\end{equation}
%Hence, from \eqref{eq16} we have
%\begin{align}
%\nonumber\sigma^2_{n,k}&=\E\left(\int_{W_n} a_k(t)H_k(\tilde{Y}(t))\nu_d(dt)\right)^2 \\
%\nonumber&=\int_{W_n}\int_{W_n} a_k(t) a_k(s)\E [H_k(\tilde{Y}(t))H_k(\tilde{Y}(s))] \nu_d(dt) \nu_d(ds)\\
%\label{eq162} &= k! \int_{W_n}\int_{W_n} a_k(t) a_k(s)\rho^k(t,s)\nu_d(dt) \nu_d(ds),\, k\in \N.
%\end{align}
{So, it follows from \eqref{eq161}  that $\Var Z_n = \sigma^2_{n,1}= \int_{W_n}\int_{W_n} a_1(t)a_1(s)\rho(t,s)\nu_d(dt) \nu_d(ds),$
and  $\Var A_n  =\sum_{k=2}^{\infty} \int_{W_n}\int_{W_n}a_k(t)a_k(s)\rho^k(t,s) \nu_d(dt) \nu_d(ds).$}
Since $|\rho(t,s)|\leq 1$ and $\sum_{k=0}^\infty a_k^2(t)\leq 1$, we have 
\begin{align*}
\nonumber
&\Var A_n\leq \sum_{k=2}^{\infty} \int_{W_n}\int_{W_n}|a_k(t)||a_k(s)||\rho(t,s)|^k \nu_d(dt) \nu_d(ds)\\
%&\leq \sum_{k=2}^{\infty}(a_k)^2 \int_{W_n}\int_{W_n}\rho^2(t,s) dt ds
\nonumber&\leq  \int_{W_n}\int_{W_n}\left(\sum_{k=2}^{\infty}|a_k(t)||a_k(s)|\right) \rho^2(t,s) \nu_d(dt) \nu_d(ds)
%\nonumber&\leq 
%\int_{W_n}\int_{W_n}\left(\sum_{k=2}^{\infty}a_k^2(t)\sum_{k=2}^{\infty}a_k^2(s)\right)^{1/2} \rho^2(t,s) \nu_d(dt) \nu_d(ds)\\
%\label{ineqW}
\leq \int_{W_n}\int_{W_n} \rho^2(t,s) \nu_d(dt) \nu_d(ds).
\end{align*}
Thus, from condition \eqref{eq11} we get 
$$\frac{\Var A_n}{\Var Z_n}=\frac{\Var A_n}{\sigma_{n,1}^2}\leq \frac{\int_{W_n}\int_{W_n}\rho^2(t,s)\nu_d(dt) \nu_d(ds)}{\int_{W_n}\int_{W_n}a_1(t)a_1(s)\rho(t,s)\nu_d(dt) \nu_d(ds)}\to 0, n\to \infty.$$
It means that $\frac{A_n}{\sigma_{n,1}}$ converges to 0 in mean square sense, and hence it converges to 0 in distribution.

Thus, we obtain that {if the limiting distributions of 
$\frac{Y_n}{\sigma_{n,1}}$ and $\frac{Z_n}{\sigma_{n,1}}$ exist, then they coincide.} Combining this fact with $\E Z_n =0$ and $\frac{Z_n}{\sigma_{n,1}}\sim N(0,1),\, n\in \N,$ we obtain the statement of the Theorem.
\end{proof}
{The accuracy of normal approximation of $Y_n$ was considered in \cite{Leonenko88}.}

A Gaussian random field is positive associated (\textbf{PA}) or negative associated (\textbf{NA}) if its covariance function is non-negative or non-positive, respectively, cf. \cite{BulSh}.
In some cases, condition \eqref{eq11} can be formulated in terms of correlation function only. 
\begin{corollary}
\label{cor0-1}
Let function $f:\R\to \R$ satisfy $f(x)<u$ for all $x<0,$ $\lim_{x\to +\infty} f(x)>u,$ and \textbf{PA} random field $Y$ satisfy conditions of Theorem \ref{mainthm} with  $\inf_{t\in W_n,n\geq 1}\E Y^2(t)=\sigma_0^2>0.$ Then condition \eqref{eq11} is satisfied if
\begin{equation}
\label{cor0-2}
\lim_{n\to\infty}\frac{\int_{W_n}\int_{W_n}\rho^2(t,s)\nu_d(dt) \nu_d(ds)}{\int_{W_n}\int_{W_n}\rho(t,s)\nu_d(dt) \nu_d(ds)}= 0.
\end{equation}
\end{corollary}
\begin{proof}
Under proposed assumptions, coefficient $a_1$ can be bounded from below
$a_1(t)=\int_{\R}\mathbbm{1}\{f(\sigma(t)x)\geq u\}x\vf(x)dx=\int_0^\infty \mathbbm{1}\{f(z)\geq u\}\frac{z}{\sigma^2(t)}\vf\left(\frac{z}{\sigma(t)}\right)dz.$
Since $\lim_{x\to +\infty} f(x)>u,$ there exists $u^*>0$ such that $f(x)\geq u$ for $t>u^*.$ Therefore,
$a_1(t)\geq \int_{u^*}^\infty \frac{z}{\sigma^2(t)}\vf\left(\frac{z}{\sigma(t)}\right)dz = \vf\left(\frac{u^*}{\sigma(t)}\right)\geq \vf\left(\frac{u^*}{\sigma_0}\right).$ Thus,
\begin{equation*}
\frac{\int_{W_n}\int_{W_n}\rho^2(t,s)\nu_d(dt) \nu_d(ds)}{\int_{W_n}\int_{W_n}a_1(t)a_1(s)\rho(t,s)\nu_d(dt) \nu_d(ds)} \leq \frac{1}{\vf^2\left(\frac{u^*}{\sigma_0}\right)}\frac{\int_{W_n}\int_{W_n}\rho^2(t,s)\nu_d(dt) \nu_d(ds)}{\int_{W_n}\int_{W_n}\rho(t,s)\nu_d(dt) \nu_d(ds)}\to 0
\end{equation*}
as $n\to \infty.$
\end{proof}
In case of monotonic function $f$ we have the following corollary.
\begin{corollary}
\label{cor00}
Under the assumptions of Theorem \ref{mainthm} let $f$ be a non-decreasing function and $\{f^{-}(x)=\inf\{y\in \R, f(y)\geq x\}$ be its generalized inverse function. Then
\begin{equation}
\label{eq120}
\frac{\int_{W_n}\mathbbm{1}\{X(t)\geq u\}\nu_d(dt)-\int_{W_n}\Psi(f^{-}(u)/\sigma(t))\nu_d(dt)}{\sqrt{\int_{W_n}\int_{W_n}\vf\left(\frac{f^{-}(u)}{\sigma(t)}\right)\vf\left(\frac{f^{-}(u)}{\sigma(s)}\right)\rho(t,s)\nu_d(dt) \nu_d(ds)}}\overset{d}{\longrightarrow} N(0,1),n\to\infty,
\end{equation}
where $\Psi(u)=\int_u^{+\infty}\vf(x)dx.$ 

If $f(x)=x,x\in \R,$ $\sigma(t)=1,t\in \R^d$ and \eqref{cor0-2} holds true, then  
\begin{equation}
\label{eq121}
\frac{\int_{W_n}\mathbbm{1}\{X(t)\geq u\}\nu_d(dt)-\nu_d(W_n)\Psi(u)}{\vf(u)\sqrt{\int_{W_n}\int_{W_n}\rho(t,s)\nu_d(dt) \nu_d(ds)}}\overset{d}{\longrightarrow} N(0,1),n\to\infty.
\end{equation}
\end{corollary}
Further in the paper, we consider normalized random fields with $\sigma^2(t)=1$ and give the examples of non-stationary covariance functions $\rho$ satisfying conditions \eqref{cor0-2}.

\subsection{Stationary random fields}
In this section, we consider further applications of Theorem \ref{mainthm} and assume that the random field $Y$ is stationary. Hence, its covariance function is invariant with respect to linear translations.
\begin{corollary}
\label{ch1:cor5}
Let $Y$ be a centered stationary Gaussian random field with covariance function $C(t)=\E[ Y(t)Y(0)],t\in \R^d,$ and $\E Y^2(0)=1.$ If $\langle H_1, \mathbbm{1}\{f(\cdot)\geq u\}\rangle_{\vf} \neq 0$ and 
\begin{equation}
\label{condcor1}
\frac{\int_{\R^d} C^2(t)\nu_d(W_n \cap (W_n-t))\nu_d(dt)}{ \int_{\R^d} C(t)\nu_d(W_n \cap (W_n-t))\nu_d(dt)} \to 0, \quad n\to \infty,
\end{equation}
then for $\{X(t)=f(Y(t)),t\in\R^d\}$ it holds
\begin{equation}
\label{cor21}
\frac{\int_{W_n}\mathbbm{1}\{X(t)\geq u\}dt-\nu_d(W_n)\pr(X(0)\geq u)}{\langle H_1, \mathbbm{1}\{f(\cdot)\geq u\}\rangle_{\vf} \sqrt{\int_{\R^d} C(t)\nu_d(W_n \cap (W_n-t))\nu_d(dt)}}\overset{d}{\longrightarrow} N(0,1)
\end{equation}
as $n\to\infty.$
\end{corollary}
\begin{proof}
Consider integrals in \eqref{eq11}
\begin{align*}
\frac{\int_{W_n}\int_{W_n}C^2(t-s)\nu_d(dt) \nu_d(ds)}{\int_{W_n}\int_{W_n}C(t-s)\nu_d(dt) \nu_d(ds)}&=\frac{\int_{\R^d}\int_{\R^d}C^2(t-s)\mathbbm{1}\{t\in W_n,s\in W_n\}\nu_d(dt) \nu_d(ds)}{\int_{\R^d}\int_{\R^d}C(t-s)\mathbbm{1}\{t\in W_n,s\in W_n\}\nu_d(dt) \nu_d(ds)}\\
=\left|\begin{array}{c}
     t-s=u  \\
     s=v 
\end{array}\right|&=\frac{\int_{\R^d}\int_{\R^d}C^2(u)\mathbbm{1}\{v\in (W_n-u),v\in W_n\}\nu_d(du) \nu_d(dv)}{\int_{\R^d}\int_{\R^d}C(u)\mathbbm{1}\{v\in (W_n-u),v\in W_n\}\nu_d(du) \nu_d(dv)}\\
%&=\frac{\int_{\R^d}\int_{\R^d}C^2(u)\mathbbm{1}\{v\in (W_n-u),v\in W_n\}\nu_d(du) \nu_d(dv)}{\int_{\R^d}\int_{\R^d}C(u)\mathbbm{1}\{v\in (W_n-u),v\in W_n\}\nu_d(du) \nu_d(dv)}\\
&=\frac{\int_{\R^d} C^2(u)\nu_d(W_n \cap (W_n-u))\nu_d(du)}{ \int_{\R^d} C(u)\nu_d(W_n \cap (W_n-u))\nu_d(du)}.
\end{align*}
Thus, conditions \eqref{eq11} and \eqref{condcor1} are equivalent and the statement of the corollary follows from Theorem \ref{mainthm}.
\end{proof}

\begin{remark}
Assume that $Y$ is $\mathbf{PA}(\mathbf{NA}),$ then  condition \eqref{condcor1} can hold only if $Y$ {is long-range dependent.} Indeed, if $0<\int_{\R^d} |C(t)|dt<\infty,$ and $|C(t)|\leq 1,\, t\in \R^d,$ then $0<\int_{\R^d} C^2(t)dt<\infty.$ Using $\lim_{n\to\infty}\frac{\nu_d(W_n\cap (W_n-t))}{\nu_d(W_n)}=1,\, t\in \R^d,$ {e.g. \cite[Chapter 3, Lemma 1.2]{BulSh},} it follows that 
$$
\Delta_n:=\frac{\int_{\R^d} C^2(t)\nu_d(W_n \cap (W_n-t))\nu_d(dt)}{ \int_{\R^d} C(t)\nu_d(W_n \cap (W_n-t))\nu_d(dt)} \xrightarrow[n\to\infty]{} \frac{\int_{\R^d} C^2(t)\nu_d(dt)}{ \int_{\R^d} C(t)\nu_d(dt)}\in (0,+\infty).
$$

Measurable function $f:\R\to\R$ can be chosen arbitrarily, which means that $\E X(0)^p<+\infty$ for some $p>0$ need not be true. The normalization in limit \eqref{cor21} {must not be of CLT-type} $n^{-d/2}$ since it involves the square root of the integral of the  weighted non-integrable function $C,$ {see Example \ref{example27}.}
\end{remark}

In the multidimensional case $d>1$, the observation windows $W_n$ can extend differently in different directions. In order to parametrize the growth of $W_n$, we make some auxiliary notation. For some 
 $r_{n,l}>0,1 \leq l\leq d, n\in\N$  introduce ``normalized'' windows 
\begin{equation}
\label{Vndef}
V_n:=\left\{\left(\frac{x_1}{r_{n,1}},\ldots,\frac{x_d}{r_{n,d}}\right), (x_1,\ldots,x_d)\in W_n\right\}
\end{equation}
such that $\sup_{n\geq 1}\nu_d(V_n)<\infty.$ 
For instance, if $r_{n,l}=\sup\{|x_l|, (x_1,\ldots,x_d)\in W_n\},$ then $V_n \subseteq [-1,1]^d.$
Moreover, we assume that
there exists a ``limit'' $V$ of $V_n,$ i.e., 
$V\in \mathcal{B}(\R^d)$, and $\nu_d(W_n)\sim \nu_d(V)\prod_{l=1}^d r_{n,l},$ $\nu_d(V\Delta V_n)\to 0$ as $n\to\infty.$
\begin{corollary}
\label{ch1:cor6}
If $Y$ is a \textbf{PA}(\textbf{NA}) stationary random field with $C(t) \to 0,$ $\|t\|\to \infty,$ then condition \eqref{condcor1} holds if for some $\delta\in (0,1)$
 %Note that $\nu_d(W_n)=\nu_d(V_n)=2^d r^1_n\cdots r^d_n.$
\begin{equation}
\label{condcor2}
%\frac{\int_{V_n} \rho^2(v)dv}{ \int_{V_n} \rho(v)dv} \to 0, \quad n\to \infty.
{ \prod_{i=1}^dr_{n,i}^{-1+\delta}}\int_{|t_i|\leq r_{n,i}} C(t)\nu_d(dt) \to + \infty, \quad n\to \infty.
\end{equation}
\end{corollary}
\begin{proof}
For  $\delta\in (0,1)$ put $V_n^\delta=\prod_{i=1}^d [-r^{-\delta}_{n,i},r^{-\delta}_{n,i}].$ Changing variables $t_i=r_{n,i}s_i$ in  \eqref{condcor1} we get 
\begin{align}
\nonumber&\frac{\int_{\R^d} C^2(t)\nu_d(W_n \cap (W_n-t))\nu_d(dt)}{ \int_{\R^d} C(t)\nu_d(W_n \cap (W_n-t))\nu_d(dt)}\\
%\nonumber&=\frac{\prod_{i=1}^d r^2_{n,i} \int_{\R^d} C^2(r_{n,1}s_1,\ldots,r_{n,d}s_d)\nu_d(V_n \cap (V_n-s))\nu_d(ds)}{\prod_{i=1}^d r^2_{n,i} \int_{\R^d} C(r_{n,1}s_1,\ldots,r_{n,d}s_d)\nu_d(V_n \cap (V_n-s))\nu_d(ds)}\\
\nonumber&=\frac{\int_{V^\delta_n} C^2(r_{n,1}s_1,\ldots,r_{n,d}s_d)\nu_d(V_n \cap (V_n-s))\nu_d(ds)}{\int_{\R^d} C(r_{n,1}s_1,\ldots,r_{n,d}s_d)\nu_d(V_n \cap (V_n-s))\nu_d(ds)}\\
\nonumber&+\frac{\int_{\R^d \setminus V^\delta_n} C^2(r_{n,1}s_1,\ldots,r_{n,d}s_d)\nu_d(V_n \cap (V_n-s))\nu_d(ds)}{\int_{\R^d} C(r_{n,1}s_1,\ldots,r_{n,d}s_d)\nu_d(V_n \cap (V_n-s))\nu_d(ds)}\\
\nonumber&\leq \frac{\nu_d(V^\delta_n)\nu_d(V_n)}{ \int_{\R^d} C(r_{n,1}s_1,\ldots,r_{n,d}s_d)\nu_d(V_n \cap (V_n-s))\nu_d(ds)}\\
\nonumber&+\left(\sup_{|t_i|\geq r_{n,i}^\delta, t\in W_n }C(t)\right)\frac{ \int_{\R^d \setminus V^\delta_n} C(r_{n,1}s_1,\ldots,r_{n,d}s_d)\nu_d(V_n \cap (V_n-s))\nu_d(ds)}{\int_{\R^d} C(r_{n,1}s_1,\ldots,r_{n,d}s_d)\nu_d(V_n \cap (V_n-s))\nu_d(ds)}\\
\label{cor1:eq2}&\leq \frac{\nu_d(V^\delta_n)\nu_d(V_n)}{ \int_{\R^d} C(r_{n,1}s_1,\ldots,r_{n,d}s_d)\nu_d(V_n \cap (V_n-s))\nu_d(ds)}+\sup_{|t_i|\geq r_{n,i}^\delta, t\in W_n }C(t).
\end{align}
Take $\beta\in (0,1)$ such that $\nu_d(V_n \cap (V_n-s))\geq \frac{1}{2} \nu_d(V_n),s\in [-\beta,\beta]^d.$ Then \eqref{cor1:eq2}  can be bounded by
%Take $n$ large enough for $\nu_d(V_n \cap (V_n-s))\geq \frac{1}{2}\nu_d(V_n),s\in V_n^\delta.$ Then \eqref{cor1:eq2}  can be bounded by
\begin{align}
\nonumber&\frac{2 \nu_d(V^\delta_n)}{  \int_{[-\beta,\beta]^d} C(r_{n,1}s_1,\ldots,r_{n,d}s_d)\nu_d(ds)}+\sup_{|t_i|\geq r_{n,i}^\delta, t\in W_n }C(t)\\
\label{star}&= \frac{2  \prod_{i=1}^dr_{n,i}^{1-\delta}}{\int_{|t_i|\leq \beta r_{n,i}} C(t)\nu_d(dt)}+\sup_{|t_i|\geq r_{n,i}^\delta, t\in W_n }C(t).
\end{align}
The latter two terms tend to 0 due to conditions of the corollary.
\end{proof}
\begin{example}
\label{example27}
For simplicity, let $d=1$, $W_n=[-n,n],$ and $C(t)\sim |t|^{-\eta}$, $\eta\in (0,1)$ as $t\to+\infty.$ 
From the proof of Corollary \ref{ch1:cor6} we get with $\delta=\frac{1-\eta}{2},$ $\beta=1,$ and $r_n=n$ that 
\begin{align*}
\Delta_n &\leq \frac{2^{1+(1+\eta)/2}(n/2)^{(1+\eta)/2}}{\int_0^{n/2} C(v)dv }+\sup_{v\geq n^{(1-\eta)/2}}C(v)\\
&\sim 2^{(3-\eta)/2}(n/2)^{(\eta-1)/2}(1-\eta)+\sup_{v\geq n^{(1-\eta)/2}}C(v) \to 0, n\to \infty.
\end{align*}
One can show that the normalization in the above limit theorem can be computed as 
$
\sigma_n^2:=\int_{\R}  C(t) \nu_1 \Big( W_n\cap (W_n-t) \Big) \, dt=2\int_{0}^{2n} (2n-t)  C(t) \, dt.$
Using the symmetry of $C$ and the substitution $s=(2n-t)/(2n)$ we write
%\begin{equation*}
${\int_{0}^{2n} (2n-t)  C(t) \, dt}=
{\int_{0}^1  s  C\big(2n(1-s)\big) \, ds} \;\sim 2^{2-\eta} B(2, 1-\eta) n^{2-\eta} $
%\end{equation*}
as $n\to +\infty$,  which follows from the definition range  $p,q>0$ of the beta--function $B(p,q)$.  To summarize,  the limit \eqref{cor1} holds:
\begin{equation*}
\frac{\int_{-n}^n \mathbbm{1}\{X(t)>u\}\, dt -2n \pr(X(0)>u) }{\langle H_1, \mathbbm{1}\{f(\cdot)\geq u\}\rangle_{\vf} 2^{3/2-\eta/2} \sqrt{B(2,1-\eta)} n^{1-\eta/2}}\overset{d}{\longrightarrow} N(0,1), \quad n\to +\infty.
\end{equation*}
%Question%
Since $a_1=\langle H_1, \mathbbm{1}\{f(\cdot)\geq u\}\rangle_{\vf}\neq 0,$ the function $f$ can not be even. Additionally, we require $\E f^{1+\theta}(Y(0))<+\infty$ for some $\theta\in (0,1)$. 

As an example, {we consider
$
f(x)=\mbox {sgn} (x) \left( e^{x^2/\beta^2}-1\right), \, x\in\R,
$
for some $\beta>\sqrt{2(1+\theta)}.$ It follows that $\E X^2(0)=\E\left(e^{Y^2(0)/\beta^2}-1 \right)^2=+\infty,$ $\E X^{1+\theta}(0)<\infty.$ It can be calculated that in this case
$
a_1^{-1}=\sqrt{2\pi}(1+u)^{\beta^2/2}  
$
for $u>0$.} 
%Question%
\end{example}

In the next Lemma we check condition \eqref{condcor2} by using the asymptotic of correlation function $C$ at $\infty.$

\begin{lemma}
\label{ch1:cor7-empt}
Let conditions of Corollaries \ref{ch1:cor5}, \ref{ch1:cor6} hold true. Let  $W_n= \prod_{i=1}^d [a_{n,i},b_{n,i}]$ 
with $r_{n,i}=(b_{n,i}-a_{n,i})/2$ and there exist functions $\lambda,q:\R^d\to\R$ and such that $q \in L_1([-1,1]^d),$ $\kappa = \int_{[-1,1]^d}q(v)\prod_{i=1}^d (1-|v_i|)\nu_d(dv)>0,$ and
\begin{align*}
    \frac{C(2r_{n,1} v_1,\ldots, 2 r_{n,d} v_d)}{q(v_1,\ldots, v_d)}\sim \lambda(r_{n,1},\ldots,r_{n,d}),~n\to \infty,
\end{align*} 
uniformly on any rectangle $[a,1]^d,a\in (0,1).$ If there exists  $\delta\in (0,1)$ such that
\begin{equation}\label{cor7:eq0}\lambda (r_{n,1},\ldots,r_{n,d}) \prod_{i=1}^d {r_{n,i}^\delta} \to \infty, n\to \infty,
\end{equation}
then
for $\{X(t)=f(Y(t)),t\in\R^d\}$ it holds 
\begin{equation}
\label{cor7:eq}
\frac{\int_{W_n}\mathbbm{1}\{X(t)\geq u\}dt-\nu_d(W_n)\pr(X(0)\geq u)}{\langle H_1, \mathbbm{1}\{f(\cdot)\geq u\}\rangle_{\vf} \nu_d(W_n) \sqrt{\kappa \lambda(r_{n,1},\ldots,r_{n,d})}}\overset{d}{\longrightarrow} N(0,1)
\end{equation}
as $n\to\infty.$
\end{lemma}
\begin{proof}
First, we compute the asymptotic variance in \eqref{cor21}. From the proof of Corollary \ref{ch1:cor6} we get that
\begin{align}
\nonumber    &\int_{\R^d} C(t)\nu_d(W_n \cap (W_n-t))\nu_d(dt) \\
\nonumber    &=\prod_{i=1}^d r_{n,i}^2 \int_{\R^d} C(r_{n,1} s_1,\ldots, r_{n,d} s_d)\nu_d([-1,1]^d \cap ([-1,1]^d-s))\nu_d(ds)\\
%   &=\prod_{i=1}^d r_{n,i}^2 \int_{[-2,2]^d} C(r_{n,1} s_1,\ldots, r_{n,d} s_d)\prod_{i=1}^d (2-|s_i|)\nu_d(ds)\\
 \label{cor7:eq3}   &=4^d \prod_{i=1}^d r_{n,i}^2 \int_{[-1,1]^d} C(2r_{n,1} v_1,\ldots, 2 r_{n,d} v_d)\prod_{i=1}^d (1-|v_i|)\nu_d(dv)\\
\nonumber    &=4^d \prod_{i=1}^d r_{n,i}^2 \left(\int_{V_n^\delta}+\int_{[-1,1]^d\setminus V_n^\delta}\right) C(2r_{n,1} v_1,\ldots, 2 r_{n,d} v_d)\prod_{i=1}^d (1-|v_i|)\nu_d(dv),
\end{align}
where  $V_n^\delta=\prod_{i=1}^d [-r^{-\delta}_{n,i},r^{-\delta}_{n,i}]$ and $\delta\in (0,1)$ is from condition \eqref{cor7:eq0}.
We can bound the first integral by {
%\begin{align*}
$4^d \prod_{i=1}^d r_{n,i}^2 \int_{V_n^\delta}C(2r_{n,1} v_1,\ldots, 2 r_{n,d} v_d)\prod_{i=1}^d (1-|v_i|)\nu_d(dv)\leq 4^d \prod_{i=1}^d r_{n,i}^2 \nu_d(V_n^\delta)=8^d \prod_{i=1}^d r_{n,i}^{2-\delta}.$}
%\end{align*}
Consider the second integral
\begin{align*}
&4^d \prod_{i=1}^d r_{n,i}^2 \int_{[-1,1]^d\setminus V_n^\delta} C(2r_{n,1} v_1,\ldots, 2 r_{n,d} v_d)\prod_{i=1}^d (1-|v_i|)\nu_d(dv)\\
=&4^d \prod_{i=1}^d r_{n,i}^2 \int_{[-1,1]^d} \frac{C(2r_{n,1} v_1,\ldots, 2 r_{n,d} v_d)}{q(v_1,\ldots, v_d)}\mathbbm{1}\{r^{-\delta}_{n,i}\leq |v_i|\leq 1,1\leq i\leq d\}\\
\times& q(v)\prod_{i=1}^d (1-|v_i|)\nu_d(dv)
\underset{n\to\infty}{\sim} 4^d \lambda(r_{n,1},\ldots r_{n,d}) \prod_{i=1}^d r_{n,i}^2  \int_{[-1,1]^d}q(v)\prod_{i=1}^d (1-|v_i|)\nu_d(dv).
\end{align*}
Therefore, from condition \eqref{cor7:eq0} it  follows that {
$
\int_{\R^d} C(t)\nu_d(W_n \cap (W_n-t))\nu_d(dt) \sim$ $ 4^d \kappa \lambda(r_{n,1},\ldots r_{n,d})  \prod_{i=1}^d r_{n,i}^2
=\kappa \lambda(r_{n,1},\ldots r_{n,d})  \nu_d^2(W_n)
$ as $n\to \infty.$}
Condition \eqref{condcor2} from Corollary \ref{ch1:cor6} is checked using asymptotic relation \eqref{cor7:eq0}  similarly to \eqref{cor7:eq3}
\end{proof}
Let us illustrate the last Lemma by the following example.
\begin{example}
Let $d=3,$  $W_n=[0,n]\times[0,n^\gamma]\times[0,c],$ $c,\gamma>0,$ then $r_{n,1}=n/2,r_{n,2}=n^\gamma/2,r_{n,3}=c/2$ and $\nu_3(W_n)=cn^{1+\gamma}.$ Consider the covariance function  $C(x,y,z)=e^{-|z|}(1+x^2+y^2)^{-\alpha},(x,y,z)\in\R^3,$ with $\alpha\in \left(0,\frac{1}{2}\right).$

In the case $\gamma\in (0,1),$ we put $q(x,y,z)=|x|^{-2\alpha}e^{-c|z|},(x,y,z)\in\R^3$ and $\lambda(r_{n,1},r_{n,2},r_{n,3})=n^{-2\alpha},n\geq 1.$ Indeed, due to Lemma \ref{ch1:cor7-empt},
$$
    \frac{C(2r_{n,1} x,2 r_{n,2} y, 2 r_{n,3} z)}{q(x,y,z)}=\frac{(1+ x^2 n^2 +  y^2 n^{2\gamma})^{-\alpha}e^{-c|z|}}{|x|^{-2\alpha}e^{-c|z|}}\sim n^{-2\alpha},\,n\to \infty.
$$
Then $q\in L_1([-1,1]^3)$ and 
\begin{equation}
\label{exmp29:eq}
\frac{\int_{W_n}\mathbbm{1}\{X(t)\geq u\}dt-c n^{1+\gamma}\pr(X(0)\geq u)}{\langle H_1, \mathbbm{1}\{f(\cdot)\geq u\}\rangle_{\vf} c n^{1+\gamma-\alpha} \sqrt{\kappa} }\overset{d}{\longrightarrow} N(0,1),\,n\to\infty.
\end{equation}
%where $kappa=$

Let $\gamma=1,$ then $q(x,y,z)=(x^2+y^2)^{-\alpha} e^{-c|z|},(x,y,z)\in\R^3$ and $\lambda(r_{n,1},r_{n,2},r_{n,3})=n^{-2\alpha},n\geq 1.$ Indeed, by Lemma \ref{ch1:cor7-empt},
$$
    \frac{C(2r_{n,1} x,2 r_{n,2} y, 2 r_{n,3} z)}{q(x,y,z)}=\frac{(1+ x^2 n^2 +  y^2 n^{2})^{-\alpha}e^{-c|z|}}{(x^2+y^2)^{-\alpha}e^{-c|z|}}\sim n^{-2\alpha},n\to \infty.
$$
Then $q\in L_1([-1,1]^3)$  and \eqref{exmp29:eq} holds true.

If $\gamma>1,$ then $q(x,y,z)=|y|^{-2\alpha} e^{-c|z|},(x,y,z)\in\R^3$ and $\lambda(r_{n,1},r_{n,2},r_{n,3})=n^{-2\gamma\alpha},n\geq 1.$ Then $q\in L_1([-1,1]^3)$  and 
\begin{equation*}
\frac{\int_{W_n}\mathbbm{1}\{X(t)\geq u\}dt-c n^{1+\gamma}\pr(X(0)\geq u)}{\langle H_1, \mathbbm{1}\{f(\cdot)\geq u\}\rangle_{\vf} c n^{1+\gamma(1-\alpha)} \sqrt{\kappa} }\overset{d}{\longrightarrow} N(0,1),n\to\infty.
\end{equation*}
\end{example}

%\begin{remark}
%Let $W_n=[a_n^1,b_n^1]\times\cdots \times[a_n^d,b_n^d].$ 
%\begin{equation}
%\label{rem1}
%\frac{\int_{[a,a+r_n]^d}\mathbbm{1}\{X(t)\geq u\}dt-r_n^d\Psi(u)}{\vf(u)\sqrt{\int_{W_n}\int_{W_n}C(t',s')\rho(|t_m-s_m|)dt ds}}\stackrel{d}{\to} N(0,1),n\to\infty.
%\end{equation}
%\end{remark}

\subsection{Spatio-temporal random fields}
In this section, we apply Theorem \ref{mainthm} to random fields which can posses different properties with respect  to the space and time coordinates.
First, we consider a separable covariance function with a stationary time-component.
\begin{theorem}
\label{cor1}
Let $\{Y(x,t),\,x\in \R^d, t\in \R\}$ be a centered \textbf{PA}(\textbf{NA})  Gaussian random field with covariance function $ \Cov(Y(x,t),Y(y,s))=C(x,y)\tilde{C}(|t-s|),$ $x,y\in \R^d,$  $s,t\in \R$ and $\E Y^2(x,t)=1.$ Assume that $\tilde{C}$ is non-negative and $\tilde{C}(r)\to 0,$ $r\to+\infty.$ Let $W_n=U_n\times(a_n,b_n)$ be a sequence of Borel sets such that %$a_n<0, b_n>0$ for all $n\geq 1$ with 
$r_n:=b_n-a_n \to +\infty, n\to\infty,$  {$0<c_1\leq \nu_{d}(U_n)\leq c_2<+\infty$ and
$\tilde{\kappa}=\lim_{n\to \infty}\int_{U_n}\int_{U_n}C(x,y)\nu_d(dx) \nu_d(dy)\in (0,+\infty).$} If for some $\delta\in (0,1)$
\begin{equation}
\label{condcor}
\frac{1}{r^\delta}\int_0^r\tilde{C}(v)dv \to \infty, \quad r \to \infty,
\end{equation}
then 
\begin{equation}
\label{stcor}
\frac{\int_{a_n}^{b_n}\int_{U_n}\mathbbm{1}\{X(x,t)\geq u\}\nu_d(dx)dt-\nu_d(U_n)(b_n-a_n)\pr(X(0)\geq u)}{\langle H_1, \mathbbm{1}\{f(\cdot)\geq u\}\rangle_{\vf}\sqrt{2\tilde{\kappa}\int_{0}^{r_n}\tilde{C}(s)(r_n-s)ds}}
\end{equation}
tends to $N(0,1)$ in distribution as $n\to\infty.$ Here $0\in \R^{d+1}.$
\end{theorem}
\begin{proof}
We check the conditions of Theorem \ref{mainthm}. First, it is obvious that $|C(x,y)|\leq 1$ and $a_1(t){=a_1}=\langle H_1, \mathbbm{1}\{f(\cdot)\geq u\}\rangle_{\vf}$ in \eqref{eq111}.
%Since $\rho(|t_m-t_m|)=1,$ we have 
%$|C(t',s')|=|C(t',s')\rho(|t_m-t_m|)|=|\Cov(Y(t),Y(s_1,\ldots,s_{m-1},t_m,s_{m+1},\ldots,s_d))|\leq 1.$
Rewrite the limit in \eqref{eq11} {multiplied by $a_1^2$} as
\begin{align}
\nonumber&\lim_{n\to\infty}\left|\frac{\int_{U_n}\int_{U_n}C^2(x,y)\nu_d(dx) \nu_d(dy) \int_{a_n}^{b_n}\int_{a_n}^{b_n}\tilde{C}^2(|t-s|)dt ds}{\int_{U_n}\int_{U_n}C(x,y)\nu_d(dx) \nu_d(dy) \int_{a_n}^{b_n}\int_{a_n}^{b_n}\tilde{C}(|t-s|)dt ds}\right|\\
\nonumber&=\lim_{n\to\infty}\frac{\int_{U_n}\int_{U_n}C^2(x,y)\nu_{d}(dx)\nu_{d}(dy)}{\int_{U_n}\int_{U_n}|C(x,y)|\nu_{d}(dx)\nu_{d}(dy)}\frac{r_n^2\int_{0}^{1}\tilde{C}^2(r_nt)(1-|t|)dt }{r_n^2\int_{0}^{1}\tilde{C}(r_n t)(1-|t|)dt}\\
\label{corlim}&\leq \lim_{n\to\infty} \frac{\int_{0}^{r_n}\tilde{C}^2(s)(r_n-s)ds }{\int_{0}^{r_n}\tilde{C}(s)(r_n-s)ds}.
\end{align}
%Denote $r=b_n-a_n.$ Changing variables $v=t_m-s_m,$ $w=t_m+s_m$ in  \eqref{corlim} we get
%\begin{align*}\frac{\int_{-(b_n-a_n)}^{b_n-a_n}\int_{2a_n+|v|}^{2b_n-|v|}\tilde{\rho}^2(|v|) dw dv}{\int_{-(b_n-a_n)}^{b_n-a_n}\int_{2a_n+|v|}^{2b_n-|v|}\tilde{\rho}(|v|) dw dv}&=\frac{\int_{-r}^{r}\tilde{\rho}^2(|v|)(r-|v|)dv }{\int_{-r}^{r}\tilde{\rho}(|v|)(r-|v|)dv}\\
%&=\frac{\int_{0}^{r}\tilde{\rho}^2(v)(r-v)dv }{\int_{0}^{r}\tilde{\rho}(v)(r-v)dv}.
%\end{align*}
Take $\delta\in (0,1)$ such that \eqref{condcor} holds true. Then 
\begin{align*}
\int_{0}^{r_n}\tilde{C}^2(v)(r_n-v)dv &= \int_0^{r_n^\delta} \tilde{C}^2(v)(r_n-v)dv + \int_{r_n^\delta}^{r_n} \tilde{C}^2(v)(r_n-v)dv\\
&\leq \int_0^{r_n^\delta} (r_n-v)dv + \left(\sup_{v\geq r_n^\delta}\tilde{C}(v)\right) \int_{r_n^\delta}^{r_n} \tilde{C}(v)(r_n-v)dv. 
\end{align*}
So, the limit in \eqref{corlim} is bounded from above by
\begin{equation}
\label{ineq11} \frac{r_n^{1+\delta}}{\int_0^{r_n} \tilde{C}(v)(r_n-v)dv}+ \frac{\int_{r_n^\delta}^{r_n} \tilde{C}(v)(r_n-v)dv}{\int_{0}^{r_n}\tilde{C}(v)(r_n-v)dv } \sup_{v\geq r_n^\delta}\tilde{C}(v)
\leq \frac{2^{1+\delta}(r_n/2)^{\delta}}{\int_0^{r_n/2} \tilde{C}(v)dv }+\sup_{v\geq r_n^\delta}\tilde{C}(v).
\end{equation}
From \eqref{condcor} we get that $r_n^{\delta}\left(\int_0^{r_n} \tilde{C}(v)dv \right)^{-1}\to 0,r_n\to+\infty.$
Moreover, we have $\sup_{v\geq r_n^\delta}\tilde{C}(v)\to 0.$ Hence, the limit in \eqref{ineq11} is equal to 0 and condition \eqref{eq11} of Theorem \ref{mainthm} is satisfied. So, the asymptotic normality of \eqref{stcor} follows from Theorem~\ref{mainthm}.
\end{proof}

Applying the approach of the above proof, we get the following evident corollary for non-separable space-time covariance functions.
\begin{corollary}
\label{cor11}
 Let $W_n=U_n\times(a_n,b_n)$ be a sequence of Borel sets such that $r_n:=b_n-a_n \to +\infty, n\to\infty,$ and {$0<c_1\leq \nu_{d}(U_n)\leq c_2<+\infty.$} Let $\{Y(x,t),x\in \R^d, t\in \R\}$ be a centered \textbf{PA}(\textbf{NA})  Gaussian random field with covariance function satisfying $$ d_1\tilde{C}(|t-s|) \leq \Cov(Y(x,t),Y(y,s))\leq d_2\tilde{C}(|t-s|),\, x,y\in \R^d,\, s,t\in \R,$$ where $d_1,d_2>0$ and $\E Y^2(x,t)=1.$ Assume that $\tilde{C}$ is non-negative and $\tilde{C}(r)\to 0,$ $r\to+\infty.$  {If for some $\delta\in (0,1)$
%\begin{equation}
%\label{condcor11}
$r^{-\delta}\int_0^r\tilde{C}(v)dv \to \infty, \, r \to \infty,$
%\end{equation}
then the sequence \eqref{stcor} is asymptotically standard normal as $n\to \infty.$}
\end{corollary}

%which illustrate above results are space–time covariance models. 

%The very last corollary is applicable to the class of non-separable asd

\begin{example}
Consider a stationary Gaussian random field $\{Y(\mathbf{x},t),\mathbf{x}\in \R^d,$ $t>0\}$ on observation windows $W_n=[0,1]^{d}\times [0,n], n\geq 1$  with the covariance function 
$$C_G(\mathbf{x},t)=\frac{1}{|t|^{2\alpha}+1}\exp\left(-\frac{\|x\|^{2\gamma}}{(|t|^{2\alpha}+1)^\gamma}\right),\,\mathbf{x}\in \R^d,\,t\in \R,$$
where parameters $\alpha\in (0,1]$ and $\gamma \in (0,1]$ govern the smoothness of the purely temporal and purely spatial covariance, see \cite{gneiting}.

Since $Y$ is stationary, we can apply  Lemma  \ref{ch1:cor7-empt} with $q(\mathbf{x},t)=|t|^{-2\alpha}$ and $\lambda(k_1,\ldots,k_{d+1})=k_{d+1}^{-2\alpha}.$
Indeed, $\frac{|t|^{2\alpha}}{|nt|^{2\alpha}+1}\exp\left(-\frac{\|x\|^{2\gamma}}{(|nt|^{2\alpha}+1)^\gamma}\right)\sim \frac{1}{n^{2\alpha}},n\to \infty.$ Therefore,  $\kappa= \int_{[-1,1]^d}\prod_{i=1}^d (1-|v_i|)\nu_d(dv)\int_{-1}^{1}|v|^{-2\alpha} (1-|v|)dv= \frac{1}{(1-2\alpha)(1-\alpha)}$ for $\alpha\in (0,1/2)$  and \eqref{cor7:eq}  becomes
\begin{equation*}
\frac{\int_{W_n}\mathbbm{1}\{X(\mathbf{x},t)\geq u\}\nu_d(d\mathbf{x})dt-n \pr(X(0)\geq u)}{\langle H_1, \mathbbm{1}\{f(\cdot)\geq u\}\rangle_{\vf} n^{1-\alpha} \sqrt{\kappa} }\overset{d}{\longrightarrow} N(0,1),n\to\infty.
\end{equation*}

To illustrate Corollary \ref{cor11}, we consider the following example of a non-separable covariance function.
\end{example}
\begin{example} Let $\{\tilde{Y}(x,t),x,t\in\R\}$ be a stationary Gaussian random field with covariance function $C_G$ as above. We make a quadratic transformation of the spatial coordinates to get a non-stationary random field $X=\{f(\tilde{Y}(\mathbf{x}^{\top} A\mathbf{x},t)), \mathbf{x}\in \R^d,$ $t>0\},$ where $A$ is a symmetric real $d\times d$-matrix.
Then
$$ \E [\tilde{Y}(\mathbf{x}^{\top} A\mathbf{x},t)Y(\mathbf{y}^{\top} A\mathbf{y},s)]=\frac{1}{|t-s|^{2\alpha}+1}\exp\left(-\frac{|\mathbf{x}^{\top} A\mathbf{x}-\mathbf{y}^{\top} A\mathbf{y}|^{2\gamma}}{(|t-s|^{2\alpha}+1)^\gamma}\right)$$ and we can apply Corollary \ref{cor11} due to
$$  \frac{e^{-(2d)^\gamma\|A \|_1^{2\gamma}}}{|t-s|^{2\alpha}+1} \leq \E [\tilde{Y}(\mathbf{x}^{\top} A\mathbf{x},t)Y(\mathbf{y}^{\top} A\mathbf{y},s)]\leq \frac{1}{|t-s|^{2\alpha}+1},$$ for all $\mathbf{x},\mathbf{y}\in [0,1]^d, t,s>0,$ where $\|A\|_1=\max_{1\leq j\leq d}\sum_{i=1}^d |a_{ij}|$ is the matrix norm induced by the sum norm $\|\cdot\|_1$ of $\R^d.$
\end{example}

\section{Non-Gaussian Limits}
Theorem \ref{mainthm} has been proved for the subordinated Gaussian random fields with  $\rank \mathbbm{1}\{f(\cdot)  \geq u\} =1.$ Now we consider the limiting behaviour of  general integral functionals of Gaussian random fields. In this case, the limiting distribution in the corresponding limit theorem is a multiple Wiener-It\^{o} integral.
\subsection{Non-stationary random fields}
Due to Kahrunen's theorem,  a Gaussian random field $Y=\{Y(t),t\in \R^d\}$ has the spectral representation
\begin{equation}
\label{Ydef}
Y(t)=\int_{\R^d}h(x,t)M(dx),t\in \R^d,
\end{equation}
where $M$ is a symmetric complex-valued Gaussian random measure with the Lebesgue control measure and $h(\cdot,t)\in L^2(\R^d)$ for all $t\in \R^d.$
\begin{remark}
If a function $h:\R^m\times \R^m\to \mathbb{C}$ is symmetric, $h(-x,\cdot)=\overline{h(x,\cdot)},x\in \R^d,$ then the random field $Y$ in \eqref{Ydef} is real-valued due to the symmetry of $M.$ In general,  $M$ can have  some other control measure $m(\cdot).$ If $m(\cdot)$ is absolutely continuous with  density $\mu:\R^d\to \R_+$, then  
we can rewrite $Y(t)=\int_{\R^d}h_s(x,t)M_s(dx),$ where $h_s(x,\cdot)=h(x,\cdot)\sqrt{\mu(x)},x\in\R^d$ and $M_s(A)=(m(A))^{-1/2}M(A),A\in \mathcal{B}(\R^d)$ is a Gaussian random measure with  the Lebesgue control measure.
\end{remark}
Moreover, we assume that the system $\{h(\cdot,t),t \in \R^d\}$ is complete in $L^2(\R^d).$ Here, $h(x,t)$ can be considered as a linear filter applied to a Gaussian random field  measure $M.$
Consequently, the covariance function $\rho:\R^d\times \R^d\to \R$ of $Y$ has a representation
$$\rho(t,s)=\int_{\R^d} h(x,t)\overline{h(x,s)}\nu_d(d x),\,t,s\in \R^d,$$
which is real-valued if $h(\cdot,t)$ is symmetric for all $t\in \R^d.$

For $x_j=(x_{j,1},\ldots,x_{j,d})\in \R^d,$ $j=1,\dots,m,$ introduce 
\begin{equation}
\label{def:limit}I_n(x_1,\ldots,x_m):= \sigma^{-1}_{n,m}\prod_{l=1}^d r_{n,l}^{-m/2}\int_{W_n}\prod_{j=1}^m h\left(\left(\frac{x_{j,1}}{r_{n,1}},\ldots, \frac{x_{j,d}}{r_{n,d}}\right), t\right) \nu_d(dt),
\end{equation}
where  $\sigma_{n,m}^2=m! \int_{W_n}\int_{W_n}\rho^m(t,s)\nu_d(dt)\nu_d(ds).$
Before stating the main results of this section we study the limit of $I_n$ as $n\to \infty.$ We illustrate it with the help of the following example on a filtered Gaussian random field.
\begin{example}
\label{exmpl1}
Consider the special case $d=1,m =2,$ $W_n=[0,n],$ and the filter  $$h(x,t)=\frac{1}{\sqrt{2\Gamma(1-2\alpha)}}\frac{e^{i g(tx)-|x|/2}}{|x|^\alpha},x\in \R, t\geq 0,$$
 with $\alpha\in \left(\frac{1}{4},\frac12\right).$ Let $g\in C^2(\R)$ be increasing, odd, $g'(x)\geq c>0,$ $x\in \R$ and $\{g(x),x\geq 0\}$ be convex. Then the corresponding filtered  Gaussian random process $Y$ is given by {$Y(t)=(2\Gamma(1-2\alpha))^{-1/2}\int_\R {\exp({i g(tx)-|x|/2})}|x|^{-\alpha}M(dx),t\geq 0.$}
%We put $h(x,t)=\frac{\tilde{h}(x,t)}{\|\tilde{h}(\cdot,t)\|},x\in \R,t\geq 0.$ Compute the norm 
%$$\|\tilde{h}(\cdot,t)\|^2=2\int_{\R_+}\frac{e^{-2t|x|}}{|x|^{2\alpha}}dx=\frac{2\Gamma(1-2\alpha)}{(2t)^{1-2\alpha}},t\geq 0,$$ which is finite for $\alpha<\frac{1}{2}.$
The covariance function of $Y$ then equals
\begin{align}
\nonumber    \rho(t,s)&=\int_\R h(x,t)\overline{h(x,s)}dx=\frac{1}{2\Gamma(1-2\alpha)}\int_\R e^{i(g(tx)-g(sx))}\frac{e^{-|x|}}{|x|^{2\alpha}}dx.
%\nonumber    &=\frac{(2t)^{1/2-\alpha}(2s)^{1/2-\alpha}}{2\Gamma(1-2\alpha) }\frac{2\Gamma(1-2\alpha)}{(t+s)^{1-2\alpha}}\\
%\label{exm2:cov}    &=\left(\frac{4ts}{(t+s)^2}\right)^{\frac{1}{2}-\alpha}=\left(1-\frac{(t-s)^2}{(t+s)^2}\right)^{\frac{1}{2}-\alpha},~t,s\geq 0.
\end{align}
Therefore, $\sigma_{n,1}^2$ equals
\begin{align}
\nonumber    \sigma_{n,1}^2&=\int_0^n\int_0^n \rho(t,s)dt ds =\frac{1}{2\Gamma(1-2\alpha)} \int_0^n\int_0^n \int_\R e^{i(g(tx)-g(sx))}\frac{e^{-|x|}}{|x|^{2\alpha}}dx dt ds\\
&=\frac{n^{2\alpha+1}}{2\Gamma(1-2\alpha)}  \int_\R \left|\int_0^1 e^{ig(ty)} dt \right|^2 \frac{e^{-|y/n|}}{|y|^{2\alpha}}dy\sim C_{\alpha,1} n^{2\alpha+1} \text{ as }n\to \infty.
\end{align}
Similarly, $\sigma_{n,2}^2$ equals
\begin{align}
\nonumber    &\sigma_{n,2}^2=2 \int_0^n\int_0^n \rho^2(t,s)dt ds \\
\nonumber    &=\frac{1}{2\Gamma^2(1-2\alpha)} \int_0^n\int_0^n \int_{\R^2} e^{i(g(tx_1)+g(tx_2)-g(sx_1)-g(sx_2))}\frac{e^{-|x_1|-|x_2|}}{|x_1 x_2|^{2\alpha}}dx_1 dx_2 dt ds\\
&=\frac{n^{4\alpha}}{2\Gamma^2(1-2\alpha)}   \int_{\R^2} \left|\int_0^1 e^{ig(ty_1)+ig(ty_2)} dt \right|^2 \frac{e^{-\frac{|y_1|+|y_2|}{n}}}{|y_1 y_2|^{2\alpha}}dy_1dy_2\sim C_{\alpha,2} n^{4\alpha}
\end{align}
 as $n\to \infty.$ Show that $C_{\alpha,1}$ and $C_{\alpha,2}$ are positive and finite due to the Lebesgue's dominated convergent theorem. Indeed, 
\begin{equation}
    \label{int:b1}2\Gamma(1-2\alpha)C_{\alpha,1}=\lim_{n\to \infty} \int_\R \left|\int_0^1 e^{ig(ty)} dt \right|^2 \frac{e^{-|y/n|}}{|y|^{2\alpha}}dy\leq \int_\R \left|\int_0^1 e^{ig(ty)} dt \right|^2 \frac{1}{|y|^{2\alpha}}dy.
\end{equation}
By integration by parts, we get 
\begin{align*}
&\left|\int_0^1 e^{ig(ty)} dt\right|\leq \left|\frac{1}{iy}\left(\frac{e^{ig(y)}}{g'(y)}-\frac{e^{ig(0)}}{g'(0)}\right)\right|+\left|\frac{1}{i y}\int_0^1\frac{y g''(ty)}{(g'(ty))^2}e^{ig(ty)}dt\right|\\
&\leq  \frac{1}{|y|}\left(\frac{1}{|g'(y)|}+\frac{1}{|g'(0)|}\right)+\frac{1}{ |y|}\left|\frac{1}{g'(y)}-\frac{1}{g'(0)}\right|\leq \frac{4}{c_1}\frac{1}{|y|}.
\end{align*}
Together with obvious inequality $\left|\int_0^1 e^{ig(ty)} dt\right|\leq 1,$ we bound the right-hand side of \eqref{int:b1} by $\int_\R \left(\left(\frac{4}{c_1}\right)^2\frac{1}{y^2} \wedge 1\right)\frac{1}{|y|^{2\alpha}}dy<\infty$ if 
$2\alpha<1.$ Similarly,
\begin{align*}
\left|\int_0^1 e^{ig(ty_1)+ig(ty_2)} dt\right|
\leq 2\left(\frac{1}{|y_1g'(y_1)+y_2g'(y_2)|}+\frac{1}{g'(0)|y_1+y_2|}\right)\wedge 1.
\end{align*}
Since $g''\geq 0,$ {we have  $|y_1g'(y_1)+y_2g'(y_2)|\geq \Big||y_1|g'(|y_1|)-|y_2|g'(|y_2|)\Big|\geq ||y_1|-|y_2||$ $\times (g'(y_2)\vee g'(y_1))\geq c_2 \Big||y_1|-|y_2|\Big|.$}
Thus,
\begin{align*}
 &   2\Gamma^2(1-2\alpha)C_{\alpha,2}\leq \int_{\R^2} \left|\int_0^1 e^{ig(ty_1)+ig(ty_2)} dt \right|^2 \frac{1}{|y_1 y_2|^{2\alpha}}dy_1dy_2 \\
& \leq \frac{4}{c_2^2}\int_{\R^2} \left(\frac{1}{(|y_1|-|y_2|)^2}\wedge 1+\frac{1}{(y_1+y_2)^2}\wedge 1\right) \frac{1}{|y_1 y_2|^{2\alpha}}dy_1dy_2<\infty
\end{align*}
if $2\alpha\in (1/2,1).$

Then sequence $I_n$ from \eqref{def:limit} has the following form 
\begin{align*}I_n(x_1,x_2)&=\frac{1}{\sigma_{n,2}n}\int_0^n e^{i g(x_1t/n)+i g(x_2t/n)} \frac{e^{-\frac{|x_1|+|x_2|}{2n}}n^{2\alpha}}{|x_1 x_2|^\alpha}dt\\
&\to \frac{1}{\sqrt{C_{\alpha,2}}}\int_0^1  \frac{e^{i g(x_1s)+i g(x_2s)}}{|x_1 x_2|^\alpha}ds=:I(x_1,x_2)
\end{align*}
in $L^2(\R^2)$ as $n\to \infty.$
Moreover, $\frac{\sigma^2_{n,2}}{\sigma^2_{n,1}}=O(n^{2\alpha-1}) $ as $n\to \infty.$

The possible examples of function $g$ are $\sinh$ and $\{x+x^{<\beta>},x\in \R\}$ for $\beta\geq 1,$ where $x^{<\beta>}=|x|^{\beta}\mathrm{sign}(x).$
\end{example}
\begin{example}
Let us consider the filter in Example \ref{exmpl1} with $g(x)=x^{<\beta>},$ $m\geq 1,$ ${h}(x,t)=\frac{1}{\sqrt{2\Gamma(1-2\alpha)}}\frac{e^{i (tx)^{<\beta>}-|x|/2}}{|x|^\alpha},x\in \R,$ and $\beta\geq 1,$ $\alpha\in \left(\frac{m-1}{2m},\frac{1}{2}\right).$  Here, it does not hold $g'(x)\geq c>0$ for all $x\in \R.$ By a substitution of variables and Jordan's lemma, it holds
\begin{align}
\label{exmpl2:b}   &\left|\int_0^1 e^{i(yt)^{<\beta>}} dt\right|\leq\frac{1}{\beta|y|}\left|\int_0^{|y|} e^{is^{\beta}} ds\right|
 %\\    &\leq \frac{1}{\beta|y|^{1/\beta}}\left(\left|\int_0^{y^\beta} \frac{\cos(s)ds}{s^{1-1/\beta}} \right|+\left|\int_0^{y^\beta} \frac{\sin(s)ds}{s^{1-1/\beta}} \right|\right)
 \leq \frac{C_{\beta}}{|y|}\wedge 1,
\end{align} 
where $C_\beta$ is a positive finite constant.
Then, similarly to Example \ref{exmpl1}, for $m\geq 1$ we have 
$\sigma_{n,m}^2\sim C_{\alpha,\beta,m} n^{(2\alpha-1)m+2}, \sigma_{n,m+1}^2\sim C_{\alpha,\beta,m+1} n^{(2\alpha-1)m+2+(2\alpha-1)}$ as $n\to \infty$ and
\begin{align*}I_n(x_1,\ldots,x_m)&=\frac{1}{\sigma_{n,m}n^{m/2}}\int_0^n \exp\left(i \frac{x^{<\beta>}_1+\cdots +x^{<\beta>}_m}{n^\beta}t^\beta\right) \frac{e^{-\frac{|x_1|+\ldots+|x_m|}{2n}}n^{m\alpha}}{|x_1 \cdots x_m|^\alpha}dt\\
&\to \frac{1}{\sqrt{C_{\alpha,\beta,m}}}\int_0^1  \frac{e^{i (x^{<\beta>}_1+\cdots +x^{<\beta>}_m)s^\beta}}{|x_1 \cdots x_m|^\alpha}ds=:I(x_1,\ldots,x_m)
\end{align*}
in $L^2(\R^m)$ as $n\to \infty.$
Indeed, $I\in L^2(\R^{m})$ because
\begin{align*}
     &\int_{\R^m} \left|\int_0^1  e^{i (x^{<\beta>}_1+\cdots +x^{<\beta>}_m)s^\beta}ds\right|^2|x_1 \cdots x_m|^{-2\alpha}dx_1\ldots dx_m\\
     &\stackrel{\eqref{exmpl2:b}}{\leq} C^2_\beta \int_{\R^m} \left(\frac{1}{|x_1+\cdots +x_m|^2}\wedge 1\right)\frac{1}{|x_1 \cdots x_m|^{2\alpha}}dx_1\ldots dx_m<\infty
\end{align*}
for $2\alpha \in \left(1-\frac{1}{m},1\right).$
\end{example}

\begin{theorem}
\label{ch2:theor}
Let $\{Y(t),t\in \R^d\}$ be a real valued centered Gaussian random field with $\E [Y(t)]^2=1$ and covariance function $\rho(t,s)=\Cov(Y(t),Y(s)),$ $t,s\in \R^d,$ which has spectral representation
\begin{equation}
\label{rho:spectr} 
\rho(t,s)=\int_{\R^d} h(x,t)\overline{h(x,s)}\nu_d(d x), \quad t,s\in \R^d,
\end{equation}
where $h(\cdot,t)\in L^2(\R^d)$ for all $t\in \R^d.$
Let $(W_n)_{n\in \N}$ be a van Hove sequence. Then for $m\in \N$ the variance $\sigma_{n,m}^2$ is equal to 
\begin{align}
\label{ch2:theor:eq}
\sigma_{n,m}^2=m!\int_{\R^{dm}}\left|\int_{W_n}\prod_{j=1}^m h(x_j,t)\nu_d(dt) \right|^2 \nu_d(d x_1)\ldots \nu_d(d x_m).
\end{align}

Assume that there exist sequences $r_{n,l},n\geq 1,$ $1 \leq l\leq d$ such that 
$I_n \to I$  in $L^2(\R^{dm})$ as $n\to\infty.$ Then 
\begin{align}
\label{limit:gen}\frac{1}{\sigma_{n,m}}\int_{W_n}H_m(Y(t))\nu_d(dt) \xrightarrow[n\to\infty]{d} \int_{\R^{dm}}^\prime I(x_1,\ldots,x_m) M(dx_1)\ldots M(dx_m).
\end{align}
\end{theorem}
\begin{proof}
Note that ${\|h(\cdot,t)\|_{L_2(\R^d)}}=1,t\in \R^d.$ For Hermite polynomials we have the following formula 
\begin{align}
H_m(Y(t))&=\int_{\R^{dm}}^\prime \prod_{j=1}^m h(x_j,t) M(dx_1)\ldots M(dx_m),\,t\in \R^d,
\end{align}
see, for example, \cite[Proposition 1.1.4]{nualart}.
Integrate both sides over $W_n$ with respect to $\nu_d(dt)$ and
%\begin{align}
%\label{fubini1}\int_{W_n}H_m(Y(t))\nu_d(dt)&=\int_{W_n} \int_{\R^{dm}}^\prime  \prod_{j=1}^m h(x_j,t)  M(dx_1)\ldots M(dx_m)\nu_d(dt).
%\end{align}
apply Fubini theorem for multiple Wiener-It\^{o} integrals (see \cite[Theorem 2.1]{Pipiras}). Indeed,  
{$\int_{W_n} \int_{\R^{dm}} \prod_{j=1}^m |h(x_j,t)|^2 \nu_d(x_1)\cdots \nu_d(x_m)\nu_d(dt)=\int_{W_n} \nu_d(dt)<\infty, n\in \N.$} So,
\begin{align*}
\frac{1}{\sigma_{n,m}}\int_{W_n}H_m(Y(t))\nu_d(dt)&=\frac{1}{\sigma_{n,m}} \int_{\R^{dm}}^\prime \int_{W_n} \prod_{j=1}^m h(x_j,t) \nu_d(dt) M(dx_1)\ldots M(dx_m).
\end{align*}

By the scaling property of Gaussian random measures, one has
$$M\left(d \left(\frac{y_{k,1}}{r_{n,1}},\ldots,\frac{y_{k,d}}{r_{n,d}}\right)\right)\stackrel{d}{=}\frac{1}{\sqrt{r_{n,1}}\cdots \sqrt{r_{n,d}}} M\left(d \left(y_{k,1},\ldots,y_{k,d}\right)\right),\, 1 \leq k\leq m.$$ After the change of variables $y_{k,l}=r_{n,l}x_{k,l}, 1 \leq k\leq m, 1 \leq l\leq d,$  the last integral equals
\begin{align*}
\frac{1}{\sigma_{n,m}\prod_{l=1}^d r_{n,l}^{m/2}} \int_{\R^{dm}}^\prime\int_{W_n}\prod_{j=1}^m h\left(\left(\frac{y_{j,1}}{r_{n,1}},\ldots, \frac{y_{j,d}}{r_{n,d}}\right), t\right) \nu_d(dt) M(dy_1)\ldots M(dy_m).
\end{align*}
Thus, we obtain 
%\begin{align}
%\label{fubini3}
{$\frac{1}{\sigma_{n,m}}\int_{W_n}H_m(Y(t))\nu_d(dt) \stackrel{d}{=}\int_{\R^{dm}}^\prime I_n(y_1,\ldots,y_m) M(dy_1)\ldots M(dy_m).$}
%\end{align}
Since $I_n\to I,n\to \infty$ in $L^2(\R^{dm}),$ we get convergence  \eqref{limit:gen}.
Applying the representation \eqref{rho:spectr} we rewrite $\sigma_{n,m}^2$ as
\begin{align*}
\nonumber\sigma_{n,m}^2%
%&=m!\int_{W_n}\int_{W_n}\rho^m(t,s)\nu_d(dt) \nu_d(ds)\\
&=m!\int_{W_n}\int_{W_n}\prod_{j=1}^m \left(\int_{\R^d} h(x_j,t)\overline{h(x_j,s)} \nu_d(d x_j)\right) \nu_d(dt) \nu_d(ds)\\
&=m!\int_{\R^{dm}}\left|\int_{W_n}\prod_{j=1}^m h(x_j,t)\nu_d(dt) \right|^2 \nu_d(d x_1)\ldots \nu_d(d x_m).
%\nonumber&=m!\int_{\R^{dm}}\left(\int_{W_n}\int_{W_n}\prod_{j=1}^m h(x_j,t)\overline{h(x_j,s)}  \nu_d(dt) \nu_d(ds) \right) \nu_d(d x_1)\ldots \nu_d(d x_m)\\
%\label{sigma:int:repr}&=
\end{align*}
\end{proof}
Under the conditions of Theorem \ref{ch2:theor} we have the following corollary.
\begin{corollary}
\label{ch2:cor01}
Let $Y=\{Y(t),t\in \R^d\}$ be a real valued centered measurable \textbf{PA} Gaussian random field with unit variance and  covariance function $\{\rho(s,t),s,t\in \R^d\}.$  Let $F\in L^2(\R,\vf)$ be  a Borel function on $\R^d$ with $m:=\rank F\geq 2$  and $b_m=\langle F,H_m\rangle_{\vf}/\sqrt{m!}.$ If $\sigma_{n,m+1}/\sigma_{n,m}\to 0,n\to \infty,$  then 
\begin{align}
\nonumber&\frac{\sqrt{m!}}{\sigma_{n,m}b_m}\left(\int_{W_n}F(Y(t))\nu_d(dt)-\nu_d(W_n)\E(F(Y(0))\right)\\
\label{ch:eq2}
&\overset{d}{\longrightarrow} \int_{\R^{dm}}^\prime I(x_1,\ldots,x_m) M(dx_1)\ldots M(dx_m),\quad n\to\infty.
\end{align}
%where function $I$ is defined in \eqref{def:limit}.
\end{corollary}

\begin{proof}
We can repeat the lines of the proof of Theorem \ref{mainthm} to get that
\begin{align*}
\int_{W_n}F(Y(t))\nu_d(dt) &= \nu_d(W_n) \E F(Y(0))\\
&+b_m \int_{W_n}\frac{H_m(Y(t))}{\sqrt{m!}} \nu_d(dt) +\sum_{k=m+1}^\infty \frac{b_k}{\sqrt{k!}} \int_{W_n}H_k(Y(t)) \nu_d(dt).
\end{align*}
Since $0\leq \rho(t,s)\leq 1,$ we have by relation \eqref{eq161} that 
\begin{align*}
\Var&\left(\sum_{k=m+1}^\infty \frac{b_k}{\sqrt{k!}} \int_{W_n}H_k(Y(t)) \nu_d(dt)\right)=\sum_{k=m+1}^\infty b_k^2  \int_{W_n} \int_{W_n}\rho^{k}(t,s)\nu_d(dt)\nu_d(ds)\\
&\leq \sum_{k=m+1}^\infty b_k^2  \int_{W_n} \int_{W_n}\rho^{m+1}(t,s)\nu_d(dt)\nu_d(ds)=\sigma^2_{n,m+1}\sum_{k=m+1}^\infty b_k^2 .
\end{align*}
Therefore, 
\begin{align}
&\lim_{n\to \infty}\frac{\int_{W_n}F(Y(t))\nu_d(dt)- \nu_d(W_n)\E(F(Y(0))}{b_m\sigma_{n,m}}
\label{ch:eq3}\stackrel{L^2(\Omega)}{=}  \lim_{n\to \infty} \frac{\int_{W_n}\frac{H_m(Y(t))}{\sqrt{m!}} \nu_d(dt)}{\sigma_{n,m}}.
\end{align}
It follows from Theorem \ref{ch2:theor}  that limit \eqref{ch:eq3} is equal to \eqref{ch:eq2} in distribution.
\end{proof}

Evidently, we can apply the above corollary to the  volumes of excursion sets by setting  $F(x)=F_u(x):=\mathbbm{1}\{f(x)\geq u\}$ for $u\in \R.$

\begin{example}
 If $f$ is symmetric, i.e., $f(x)=f_0(|x|)$ for some $f_0:\R_+\to\R_+,$ then $\langle F_u,H_1\rangle_{\vf}=\int_\R \mathbbm{1}\{f_0(|x|)\geq u\}x \varphi(x)dx=0.$ { If $f_0$ is non-decreasing and $f_0^{(-1)}(u)\geq 1$ then $\langle F_u,H_2\rangle_{\vf}=2\int_{\R_+} \mathbbm{1}\{f_0(x)\geq u\}(x^2-1) \varphi(x)dx$ is positive. In such case $m=2.$}

 Let $f(x)=f_1(|x|)\mathbbm{1}\{|x|\leq 1\}+f_2(|x|)\mathbbm{1}\{|x|> 1\},\,x\in \R,$ where $f_1:[0,1]\to \R_+$ and $f_2:[1,+\infty)\to \R_+$ are non-decreasing and non-increasing functions, respectively. Then $\E\langle F_u,H_1\rangle_{\vf}=0$, $\langle F_u,H_3\rangle_{\vf}=0,$ and
\begin{equation}
    \label{fm4}
    \langle F_u,H_2\rangle_{\vf}=-2\int_0^{f_1^{-1}(u)}(1-x^2)\varphi(x)dx+2\int_{f_2^{-1}(u)}^{+\infty}(x^2-1)\varphi(x)dx.
\end{equation}
Therefore, if $f_1^{-1}(u)$ and $f_2^{-1}(u)$ are such that the right hand side of \eqref{fm4} equals to zero, then $m=4.$
\end{example}
\subsection{Stationary random fields}

%\begin{example}
%Let us consider the random field $Y$ from Example \ref{exmpl1}. Then from \eqref{exmpl:sigm} $\sigma_{n,m}=C_{\alpha,m}n^m.$ Thus, 
%Corollary \ref{ch2:cor01} is valid with $Y$ and for any function $F:\R\to \R$ with $\rank F=m\geq 2.$
%\end{example}

%\begin{corollary}
%\label{ch2:cor}
%Let the assumptions of Theorem \ref{mainthm} are satisfied. Let $\rho(t,s)\geq 0$ and   $\{X(t)=f(Y(t)),t\in \R^d\}$ be the corresponding subordinated Gaussian random field. Denote $m:=\rank G_u.$ If $\sigma_{n,m+1}/\sigma_{n,m}\to 0,n\to \infty$ 
% then for any $u\in \R$
%\begin{align}
%\nonumber&\frac{1}{\sigma_{n,m}\E F_u(Z) H_k(Z)}\left(\int_{W_n}\mathbbm{1}\{X(t)\geq u\}\nu_d(dt)-\nu_d(W_n)\pr(X(0)\geq u)\right)\\
%\label{ch:eq2-001}
%&\stackrel{d}{\to} \int_{\R^{dm}}^\prime I(x_1,\ldots,x_m) \widetilde{W}(dx_1)\ldots \widetilde{W}(dx_m),n\to\infty,
%\end{align}
%where %$\sigma_{n,m}$ is defined in \eqref{sigma:int:repr} and 
%function $I$ is defined in \eqref{def:limit}.
%\end{corollary}

Now, consider the case of stationary random fields. Let $\{Y(t),t\in \R^d\}$ be a real-valued measurable stationary  centered Gaussian random field with $\E [Y(t)]^2=1$ and covariance function $C(t)= \Cov(Y(t),Y(0)),t\in \R^d,$ which is continuous at zero. It follows from Bochner's theorem that function $C$ has spectral representation 
\begin{equation}
\label{rho:spectr2} 
C(t)=\int_{\R^d} e^{i\langle x,t\rangle}G(dx), \,t\in \R^d,
\end{equation}
where $G$ is its spectral measure. Moreover, we assume that there exists a spectral density $g\in L^2(\R^d)$ of $G.$ %Moreover, we consider the observation windows $W_n,n\geq 1$ of the specific form. 

We take observation windows $W_n$ as in \eqref{Vndef} with "the limit" $V$ of $V_n,$ i.e., $V\in \mathcal{B}(\R^d)$ and  $\nu_d(V{\Delta} V_n)\to 0,\, n\to\infty.$ Then for any $x_k\in \R^d , 1 \leq k\leq m$ it holds $$\left|\int_{V_n}e^{i\langle x_1+\cdots+x_m,t\rangle}\nu_d(dt)-\int_{V}e^{i\langle x_1+\cdots+x_m,t\rangle}\nu_d(dt)\right|\leq \nu_d(V{\Delta} V_n)\to 0,n\to \infty,$$ uniformly in $(x_1,\ldots,x_m).$
Moreover, the functions 
\begin{align}
\label{K:def}
K_{V}(x):=\int_V e^{i\langle x,t\rangle} \nu_d(dt), \quad x\in \R^d
\end{align}
are square integrable on $\R^d$ for any  $V\in \mathcal{B}(\R^d)$ with $\nu_d(V)<\infty$ as a Fourier transform of indicator function $\mathbbm{1}(t\in V).$ 
% (see, for example, \cite{leon14}).

Let the behaviour of spectral density $g$ at 0 be {"similar" to function $\lambda:\R^d\to \R_+$ in the following sense. Assume that there exists a point-wise limit $Q:=\lim_{n\to \infty}Q_n,$ where 
\begin{align}
\label{Qn:def}Q_n(z):=\frac{1}{\lambda(r_{n,1},\ldots,r_{n,d})}\sqrt{g\left(\frac{z_{1}}{r_{n,1}},\ldots,\frac{z_{d}}{r_{n,d}}\right)},\,z=(z_1,\ldots,z_d)\in \R^d,
\end{align}
and
$d_n:={\lambda^m(r_{n,1},\ldots,r_{n,d})}{\sigma^{-1}_{n,m}} \prod_{l=1}^d r^{1-m/2}_{n,l},\, n,m\in \N.$
}
\begin{theorem}
\label{ch2:theor2}Let the field $\{Y(t),t\in \R^d\}$ be as above. Assume that there exist  functions $L,K:\R^d\to \R$ such that for any $n\in \N$ \begin{equation}
\label{cond:star}
\begin{gathered}
    Q_n(x)\leq L(x),\, x\in \R^d, \text{ and } K_V(x),K_{V_n}(x)\leq K(x),\, x\in \R^d,  \\
    \text{and }K(x_1+\ldots+ x_m)\prod_{j=1}^m L(x_j)\in L^2(\R^{md}).
\end{gathered}  
\end{equation}
 Then 
 \begin{equation}
\label{s:eq2-3}
\sigma^2_{n,m}\sim c_{m,V} m! \frac{\lambda^{2m}(r_{n,1},\ldots,r_{n,d})}{\prod_{l=1}^d r_{n,l}^{m-2}},\,n\to \infty,
\end{equation}
where 
%\begin{equation}
%\label{cor2:eq2}
{$c_{m,V}=\int_{\R^{dm}} \prod_{j=1}^m Q^2(y_j) |K_{V}(y_1+\cdots+y_m)|^2 \nu_d(dy_1)\ldots \nu_d(dy_m),$}
%\end{equation}
 and
\begin{align}
\label{limit:thm2}&\frac{\int_{W_n} H_m(Y(t))\nu_d(dt) }{\sigma_{n,m}}
 \xrightarrow[n\to\infty]{d}\\
\nonumber &\frac{1}{\sqrt{m! c_{m,V}}}\int_{\R^{dm}}^\prime \prod_{j=1}^m Q(y_j) K_{V}(y_1+\cdots+y_m) M(dy_1)\ldots M(dy_m).
\end{align}
\end{theorem}
\begin{proof}
We apply Theorem \ref{ch2:theor} with $h(x,t)=e^{i \langle x,t\rangle}\sqrt{g(x)},\,x,t \in \R^d.$ 
Functions $I_n,$ defined in \eqref{def:limit}, have the following form
\begin{align}
\nonumber &I_n(x_1,\ldots,x_m)\\
\nonumber &=\sigma^{-1}_{n,m}\prod_{l=1}^d r_{n,l}^{-m/2}\prod_{j=1}^m \sqrt{g\left(\frac{x_{j,1}}{r_{n,1}},\ldots, \frac{x_{j,d}}{r_{n,d}}\right)}\int_{W_n}\exp\left(i\left(\sum_{l=1}^d\sum_{k=1}^m \frac{x_{k,l}}{r_{n,l}}t_l\right)\right) \nu_d(dt)\\
\nonumber&=\left|t_{l}=r_{n,l}u_{l},1 \leq l\leq d \right|\\
\nonumber&=\sigma^{-1}_{n,m}\prod_{l=1}^d r_{n,l}^{1-m/2}\prod_{j=1}^m\sqrt{g\left(\frac{x_{j,1}}{r_{n,1}},\ldots, \frac{x_{j,d}}{r_{n,d}}\right)}\int_{V_n}e^{i\langle x_1+\cdots +x_m ,u\rangle} \nu_d(du)\\
\label{statIn}&=d_n\prod_{j=1}^m Q_n(x_j)K_{V_n}(x_1+\cdots+x_m).
\end{align}
Let us check that the sequence \eqref{statIn} converges in {$L^2(\R^{dm})$-sense.} The triangle inequality for the {$L^2(\R^{dm})$-norm} yields 
\begin{align}
\nonumber&\left(\int_{\R^{dm}} \left|\prod_{j=1}^m Q_n(x_j)K_{V_n}(x_1+\cdots +x_m)\right.\right.\\
\nonumber&\left.\left.-\prod_{j=1}^m Q(x_j)K_{V}(x_1+\cdots +x_m)\right|^2\nu_d(dx_1)\ldots \nu_d(dx_m)\right)^{1/2}\\
\nonumber&\leq \left(\int_{\R^{dm}} \left(\prod_{j=1}^m Q_n(x_j)-\prod_{j=1}^m Q(x_j)\right)^2 \left|K_{V}(x_1+\cdots +x_m)\right|^2\nu_d(dx_1)\ldots \nu_d(dx_m)\right)^{1/2}\\
\label{statIn2}&+\left(\int_{\R^{dm}} \prod_{j=1}^m Q_n^2(x_j) \left|K_{V{\Delta} V_n}(x_1+\cdots +x_m)\right|^2\nu_d(dx_1)\ldots \nu_d(dx_m)\right)^{1/2}.
\end{align}
The first summand in \eqref{statIn2} tends to 0 as $n\to \infty$ %due to 
%$$\prod_{j=1}^m Q^2_n(x_j)K_{V}(x_1+\ldots+x_m)  \stackrel{L^2}{\to} \prod_{j=1}^m Q^2(x_j)K_{V}(x_1+\cdots+x_m), n \to \infty,$$
%which holds 
by the Lebesgue theorem on dominated convergence. The second term in \eqref{statIn2} is bounded by
$$ \left(\int_{\R^{dm}} \prod_{j=1}^m L^2(x_j) \left|K_{V{\Delta} V_n}(x_1+\cdots +x_m)\right|^2\nu_d(dx_1)\ldots \nu_d(dx_m)\right)^{1/2},$$
which tends to 0 as $n\to \infty$ by the Lebesgue theorem as well, since $|K_{V{\Delta} V_n}(z)|\leq \nu_d(V{\Delta} V_n)\to 0, n\to\infty.$

In the case of stationary random fields, \eqref{ch2:theor:eq} has the form
\begin{equation*}
\sigma_{n,m}^2=
m!\int_{\R^{dm}}\prod_{j=1}^m g(x_j)\left|\int_{W_n}e^{i \langle x_1+\ldots+x_m ,t\rangle }  \nu_d(dt)\right|^2 \nu_d(d x_1)\ldots \nu_d(d x_m).
\end{equation*}
In the last integral, we make the change of variables $u_{l}=t_{l}/r_{n,l},$ $y_{k,l}=x_{k,l}r_{n,l},$ $1 \leq l\leq d,$  $1 \leq k\leq m.$ Then $\sigma_{n,m}^2$ rewrites  as
\begin{align}
\label{s:eq2-1}&\sigma_{n,m}^2=\frac{m!}{\prod_{l=1}^d r_{n,l}^{m-2}}\int_{\R^{dm}}\prod_{j=1}^m g\left(\frac{y_{j,1}}{r_{n,1}},\ldots,\frac{y_{j,d}}{r_{n,d}}\right)\\
\nonumber &\times\left|\int_{V_n}e^{i\langle y_1+\cdots+ y_m,u \rangle}  \nu_d(du)\right|^2 \nu_d(d y_1)\ldots \nu_d(d y_m)\\
\nonumber&\sim \frac{m!\lambda^{2m}(r_{n,1},\ldots,r_{n,d})}{\prod_{l=1}^d r_{n,l}^{m-2}}\int_{\R^{dm}}\prod_{j=1}^m Q^2\left(y_j\right)\left|K_{V}(y_1+\cdots +y_m)\right|^2 \nu_d(d y_1)\ldots \nu_d(d y_m)
\end{align}
as $ n\to \infty.$
Therefore, $d_n \to (m! c_{m,V})^{-1/2}$ as $n\to \infty$ in \eqref{statIn}. Hence, the conditions of Theorem \ref{ch2:theor} are satisfied, and the statement \eqref{limit:thm2} is proved.
\end{proof}
%\begin{remark}
%We can also rewrite $c_{m,V}$ in terms of convolutions $Q^{*k}$
%using the change of variables $z_k=y_k+\cdots+y_m,1 \leq k\leq m.$
%\begin{align*}
%&\int_{\R^{dm}}\prod_{j=1}^m Q\left(y_j\right)\left|K_{V}(y_1+\cdots +y_m)\right|^2 \nu_d(d y_1)\ldots \nu_d(d y_m)\\
%&=\int_{\R^{d(m-1)}}\left|K_{V}(z_1)\right|^2  \prod_{j=1}^{m-2} Q(z_{j}-z_{j+1})\int_{\R^d}Q(z_{m-1}-z_{m})Q(z_{m}) \nu_d(d z_m)\ldots \nu_d(d z_1)\\
%&=\int_{\R^{d(m-1)}}\left|K_{V}(z_1)\right|^2  \prod_{j=1}^{m-2} Q(z_{j}-z_{j+1}) Q^{*2}(z_{m-1})\nu_d(d z_{m-1})\ldots \nu_d(d z_1)=\ldots\\
%&=\int_{\R^{d}}\left|K_{V}(z)\right|^2 Q^{*m}(z)\nu_d(d z).
%\end{align*}
%\end{remark}
\begin{corollary}
\label{ch2:cor3}
Under the assumptions of Theorem \ref{ch2:theor2}, let  $\{X(t)=f(Y(t)),t\in \R^d\}$ be the corresponding subordinated Gaussian random field, where $f$ is a Borel function on $\R^d.$ Denote $m:=\rank G_u\geq 2,$ where $G_u(x):=\mathbbm{1}\{f(x)\geq u\}$ and $b_m=\langle G_u,H_m\rangle_{\vf}/\sqrt{m!}$ for $u\in \R.$  If
%\begin{equation}
%\label{cor3:eq1}
{${\lambda^{2}(r_{n,1},\ldots,r_{n,d})}{\prod_{l=1}^d r^{-1}_{n,l}}\xrightarrow{} 0$ as $n\to \infty,$}
%\end{equation}
then
\begin{align}
\nonumber &\frac{\int_{W_n}\mathbbm{1}\{X(t)\geq u\}\nu_d(dt)-\nu_d(W_n)\pr(X(0)\geq u)}{b_m\lambda^m(r_{n,1},\ldots,r_{n,d})\prod_{l=1}^{d}r_{n,l}^{1-m/2}}
\\
\label{limit:cor3}&\xrightarrow[n\to\infty]{d} \int_{\R^{dm}}^\prime \prod_{j=1}^m Q(y_j) K_{V}( y_1+\cdots+y_m) M(dy_1)\ldots M(dy_m).
\end{align}
\end{corollary}
\begin{proof}
The statement follows from Corollary \ref{ch2:cor01} and Theorem \ref{ch2:theor2}. Indeed,
\begin{align*}
    \frac{\sigma^{2}_{n,m+1}}{\sigma^{2}_{n,m}}&\sim \frac{c_{m+1,V,K}}{c_{m,V}}\frac{\lambda^{2m+2}(r_{n,1},\ldots,r_{n,d})}{\lambda^{2m}(r_{n,1},\ldots,r_{n,d})}\frac{\prod_{l=1}^d r^{m-2}_{n,l}}{\prod_{l=1}^d r^{m+1-2}_{n,l}}(m+1)\\
    &=(m+1)\frac{c_{m+1,V}}{c_{m,V}}\frac{\lambda^{2}(r_{n,1},\ldots,r_{n,d})}{\prod_{l=1}^d r_{n,l}}\xrightarrow{} 0,\,n\to \infty.
\end{align*}
\end{proof}

Note that if $V=[0,1]^d,$ then 
$$K_V(x)=\int_{[0,1]^d}e^{i\langle x ,t\rangle}\nu_d(dt)=\prod_{l=1}^d \frac{e^{i x_l}-1}{i x_l},\,x=(x_1,\ldots,x_d)\in \R^d,$$ and 
if $V=B_1(0)$ is a unit ball in $\R^d$ then 
$$K_V(x)=\int_{\|x\|\leq 1}e^{i\langle x ,t\rangle}\nu_d(dt)=(2 \pi)^{d/2}\frac{J_{d/2}(\|x\|)}{\|x\|^{d/2}},\, x\in \R^d,$$
where $J_{\alpha}$ is the Bessel function of the first kind of order $\alpha>-1/2,$ cf. \cite[Example~2]{leon14}.

Now consider several examples of spectral densities. The following isotropic case was considered in  \cite{Leonenko86,leon14}.   %Particularly, in \cite[Theorem 5]{leon14} it was proved that if $\alpha\in (0,d/m),$ 
%The next result follows directly from \eqref{leon:eq}.
\begin{theorem}
\label{thm:ex2}Under the conditions of Corollary \ref{ch2:cor3}, let the spectral density of $Y$ be equal to 
%\begin{equation}
%\label{g2:def}
{$g_Y(z_1,\ldots,z_d)={L_Y(\|z\|)}{\|z\|^{\alpha-d}},$} $z=(z_1,\ldots,z_d)\in \R^d,$
%\end{equation}
where $\alpha \in (0,d/{m})$ and $L_Y:\R\to \R_+$ is slowly varying at 0.
Assume that $r_{n,l}=r_n,1 \leq l\leq d.$
\begin{itemize}
    \item If $V=[0,1]^d$ then the limiting random variable in \eqref{limit:cor3} is
    \begin{align*}
    \int_{\R^{dm}}^\prime \prod_{l=1}^d \frac{e^{i(y_{1,l}+\cdots+y_{m,l})}-1}{i(y_{1,l}+\cdots+y_{m,l})} \frac{M(dy_1)\ldots M(dy_m)}{\|y_{1}\|^{\frac{d-\alpha}{2}}\cdots \|y_{m}\|^{\frac{d-\alpha}{2}}}.    
    \end{align*}
    \item If $V=B_1(0)$ then the limiting random variable in \eqref{limit:cor3} is
    \begin{align*}
    (2\pi)^{d/2}\int_{\R^{dm}}^\prime  \frac{J_{d/2}(\|y_1+\cdots+y_m\|)}{\|y_{1}+\cdots+y_{m}\|^{d/2}} \frac{M(dy_1)\ldots M(dy_m)}{\|y_{1}\|^{\frac{d-\alpha}{2}}\cdots \|y_{m}\|^{\frac{d-\alpha}{2}}}.  
    \end{align*}
\end{itemize}
\end{theorem}
\begin{proof}
The statement follows from  Corollary \ref{ch2:cor3} for functions $\lambda$ and $Q$ given by   
{$\lambda(\mathbf{r})= \|\mathbf{r}\|^{\frac{d-\alpha}{2}} L_Y^{1/2}\left({1}/{\|\mathbf{r}\|}\right),\,\mathbf{r}\in \R^d,\quad Q(x)=\|x\|^{(\alpha-d)/2},\,x\in \R^d.$} 
Inetgrability condition \eqref{cond:star} holds by \cite[Lemma 3]{leon14}.
%Then the normalizing term in \eqref{limit:cor3} equals 
%$$\frac{1}{\sqrt{m!} b_m  }\frac{ r_n^{(m/2-1)d}}{\lambda^m(r_n\cdot \mathbf{1})}=\frac{r_n^{(\alpha m-2 d)/2}}{\sqrt{m!}b_m  L_Y^{m/2}(1/r_n)},\quad n\geq 1.$$
\end{proof}

Let us now consider the anisotropic case, where spectral densities, and consequently covariance functions, are coordinate-wise products of univariate spectral densities. 
 
\begin{theorem}
\label{thm:ex1}
Let conditions of Corollary \ref{ch2:cor3} be satisfied with spectral density
\begin{equation}
\label{g1:def}g_A(z_1,\ldots,z_d)=\prod_{l=1}^d \frac{L_l(z_l)}{|z_l|^{1-\gamma_l}}, \, z=(z_1,\ldots,z_d)\in \R^d,
\end{equation}
where $\gamma_l \in (0,1/m),1 \leq l\leq d$ and $L_l,1 \leq l\leq d$ are slowly varying functions at 0.
\begin{itemize}
    \item If $V=[0,1]^d$ then the limiting random variable in \eqref{limit:cor3}  is
    \begin{align}
    \label{Hermite:sheet}
    \int_{\R^{dm}}^\prime \prod_{l=1}^d  \frac{1}{|y_{1,l}\cdots y_{m,l}|^{(1-\gamma_l)/2}}  \frac{e^{i(y_{1,l}+\cdots+y_{m,l})}-1}{i(y_{1,l}+\cdots+y_{m,l})} M(dy_1)\ldots M(dy_m).    
    \end{align}
    \item If $V=B_1(0)$ then the limiting random variable in \eqref{limit:cor3}  is
    \begin{align*}
    (2\pi)^{d/2}\int_{\R^{dm}}^\prime\prod_{l=1}^d \frac{1}{|y_{1,l}\cdots y_{m,l}|^{(1-\gamma_l)/2}} \frac{J_{d/2}(\|y_1+\cdots+y_m\|)}{\|y_{1}+\cdots+y_{m}\|^{d/2}} M(dy_1)\ldots M(dy_m).    
    \end{align*}
\end{itemize}
\end{theorem}
\begin{proof}
Similarly to the proof of Theorem \ref{thm:ex2}, functions $\lambda$ and $Q$ from Corollary \ref{ch2:cor3} have the form
{$\lambda(\mathbf{r})= \prod_{l=1}^d r_{l}^{(1-\gamma_l)/2}L_l^{1/2}\left(\frac{1}{r_{l}}\right),\,\mathbf{r}=(r_1,\ldots,r_d)\in \R^d,$
$ Q(x)=\prod_{l=1}^d |x_{l}|^{(\gamma_l-1)/2},\, x=(x_1,\ldots,x_d)\in \R^d.$}

Following the lines of the proof of \cite[Lemma 3]{leon14}, we can show that 
\begin{equation}
    \label{lem3:an}
\int_{\R^{dm}} |K(y_1+\cdots+y_m)|^2\prod_{l=1}^d \frac{\nu_d(d y_1)\cdots \nu_d(d y_m)}{|y_{1,l}|^{1-\tau_{1,l}}\cdots |y_{m,l}|^{1-\tau_{m,l}}}<\infty
\end{equation}
for $\sum_{j=1}^m \tau_{j,l} <1,1 \leq l\leq d.$ 
Thus, condition \eqref{cond:star} of Theorem \ref{ch2:theor2} is satisfied, and the required statements are true.
\end{proof}

Random variables \eqref{Hermite:sheet} have the  distribution of the marginals of Hermite sheets of order $m.$ In the case $m=2,$ those are called {\it Rosenblatt sheets}, cf. e.g. \cite{Pipiras2017,Tudor2013}. {Particularly, the representation as a Wiener-It\^{o} integral allows to compute moments of \eqref{Hermite:sheet}. The other characteristics for Rosenblatt distribution such as the L\'{e}vy measure or the characteristic function are found by another technique, cf. \cite{Leonenko2017}. The rate of convergence to the to Rosenblatt and Hermite distributions is studied in \cite{Anh19}.}

\begin{example}
For $m=2,$ we apply the results of Theorem \ref{ch2:theor2} to the spectral density 
$$g(x,y)=\tilde{g}(x,y){\left(x^2+c|y|^{\frac{2H_2}{H_1}}\right)^{-H_1/2}},\, (x,y)\in \R^2,$$
considered in \cite{surg},
where $H_1,H_2>0,$ $H_1H_2<H_1+H_2,c>0$ and $g$ is a bounded positive function with $\tilde{g}(0,0)=1.$
Put $W_n=[0,n]\times[0,n^\gamma],$ where $\gamma>0.$ Then $V_n=V=[0,1]^2$ in \eqref{Vndef} and $K(x,y)=-\frac{e^{i x}-1}{x}\frac{e^{i y}-1}{y},(x,y)\in \R^2$ in \eqref{K:def}. Similarly to \cite{surg}, the asymptotic behaviour of \eqref{limit:thm2} depends on the value of $\gamma.$ We consider several cases yielding the limit in \eqref{limit:thm2} (up to a constant factor).

Let $\gamma<H_1/H_2,$ then we have in \eqref{Qn:def} that $\lambda(n,n^\gamma)=n^{H_2 \gamma/2}$ and  $$Q_n^2(x,y)=\frac{1}{n^{H_2 \gamma}}\left(\frac{x^2}{n^2}+c \frac{|y|^\frac{2H_2}{H_1}}{n^{\frac{2H_2}{H_1}\gamma}}\right)^{-\frac{H_1}{2}}\tilde{g}\left(\frac{x}{n},\frac{y}{n^\gamma}\right)\xrightarrow[n \to \infty]{} c^{-\frac{H_1}{2}}|y|^{-H_2}=Q^2(x,y) $$
point-wise. Thus, the limit in \eqref{limit:thm2} reads
$$\int_{\R^{4}}^\prime |y_1 y_2|^{-\frac{H_2}{2}}\frac{e^{i(y_{1}+y_{2})}-1}{(y_{1}+y_{2})}\frac{e^{(x_{1}+x_{2})}-1}{i(x_{1}+x_{2})} M(dx) M(dy).$$

Let $\gamma=H_1/H_2,$ then $\lambda(n,n^\gamma)=n^{\frac{H_1}{2}}$ and $$Q_n^2(x,y)=\frac{1}{n^{H_1 }}\left(\frac{x^2}{n^2}+c \frac{y^\frac{2H_2}{H_1}}{n^{\frac{2H_2}{H_1}\gamma}}\right)^{-\frac{H_1}{2}}\tilde{g}\left(\frac{x}{n},\frac{y}{n^\gamma}\right)\longrightarrow \left(x^2+c  y^\frac{2H_2}{H_1}\right)^{-\frac{H_1}{2}}$$
as $n \to \infty$ point-wise. Here,  the limit in \eqref{limit:thm2} equals (up to a constant)
$$\int_{\R^{4}}^\prime \prod_{l=1}^2\left(x_l^2+c  y_l^\frac{2H_2}{H_1}\right)^{-\frac{H_1}{4}} \frac{e^{i(y_{1}+y_{2})}-1}{(y_{1}+y_{2})}\frac{e^{(x_{1}+x_{2})}-1}{i(x_{1}+x_{2})} M(dx) M(dy).$$

If $\gamma>H_1/H_2,$ then we have $\lambda(n,n^\gamma)=n^{\frac{H_1}{2} }$ in \eqref{Qn:def} and $$Q_n^2(x,y)=\frac{1}{n^{H_1}}\left(\frac{x^2}{n^2}+c \frac{y^\frac{2H_2}{H_1}}{n^{\frac{2H_2}{H_1}\gamma}}\right)^{-\frac{H_1}{2}}\tilde{g}\left(\frac{x}{n},\frac{y}{n^\gamma}\right)\longrightarrow |x|^{-H_1}$$
as $n \to \infty$ point-wise. This yields the limit in \eqref{limit:thm2}:
$$\int_{\R^{4}}^\prime |x_1 x_2|^{-\frac{H_1}{2}}\frac{e^{i(y_{1}+y_{2})}-1}{(y_{1}+y_{2})}\frac{e^{i(x_{1}+x_{2})}-1}{(x_{1}+x_{2})} M(dx) M(dy).$$
%\begin{theorem}
%\label{thm:ex3}
%Let conditions of Corollary \ref{ch2:cor3} are satisfied and the spectral density of $Y$ has the following form
%\begin{equation}
%\label{g3:def}g_3(z_1,\ldots,z_d)=L_3(z_1,\ldots,z_d) |z_1|^{-\beta_1}\|z_2^2+\cdots+ z_d^2\|^{-\beta_2}, 
%\end{equation}
%where $\beta_1\in [0,1/2),$ $\beta_2\in [0,\frac{1}{2(d-1)}).$ 
\end{example}

\section{Examples}
In this section, we provide further applications of our results to some well known classes of stationary and non-stationary random fields.

\subsection{Ergodic theorem for random volatility models}
Let $\{Y(t),t\in \R^d\}$ be a measurable centered stationary Gaussian random field with covariance function $C:\R^d \to \R_+,$ $C(0)=1.$ Let $\xi$ be a non-negative random variable, independent of $Y.$ Introduce a random volatility field by $X(t)=\xi Y(t), t\in \R^d.$
{This notion comes from applications in financial time series for the case $d=1,$ see \cite[Part II ]{andersen2009handbook}.}
Particularly, if $\xi$ is a non-negative $\alpha$-stable random variable, then $\E X^2(t)=+\infty.$ Denote by $\Psi(\cdot)$ the tail probability function of $N(0,1).$
 
 Let us consider the asymptotic behaviour of volumes of excursion sets for random field $X.$
%Let $\xi$ be symmetric $\alpha$-stable random variable with $\E e^{is \xi}=\exp(-\tau^\alpha |s|^\alpha), \tau>0.$ independent of $Y.$ 

\begin{theorem}
\label{RV:thm} Let $X=\{X(t),t\in \R^d\}$ be a random volatility field as above, such that 
\begin{equation}
    \label{RV:eq1}
    \frac{1}{\nu_d^2(W_n)}\int_{\R^d}|C(t)|\nu_d(W_n\cap(W_n-t))\nu_d(dt)\xrightarrow[n\to\infty]{}0
\end{equation}
for a sequence of observation windows $\{W_n\}_{n=1}^\infty$ growing in van Hove sense. Then 
\begin{equation*}
\frac{1}{\nu_d(W_n)}\int_{W_n}\mathbbm{1}(X(t)>u)\nu_d(dt)\xrightarrow[n\to\infty]{d} \Psi\left(\frac{u}{\xi}\right).
\end{equation*}
\end{theorem}
\begin{proof}
As in Theorem \ref{mainthm}, we write the Hermite expansion of function $F(x,\sigma)=\mathbbm{1}\{x \sigma\geq u\}$ with Fourier coefficients
$a_k(\sigma)=\frac{1}{\sqrt{k!}}\langle F(\cdot,\sigma)H_k \rangle_{\vf}.$
Then {
$\int_{W_n}F(Y(t),\xi)\nu_d(dt)=\sum_{k=0}^\infty  \int_{W_n}a_k(\xi)\frac{H_k(Y(t))}{\sqrt{k!}}\nu_d(dt)$ a.s.}
The summands with $k\geq 1$ are centered and  uncorrelated due to independence of $\xi$ and $Y.$  The variance of each summand equals
\begin{align*}
    &\E \left[ a_k(\xi)\int_{W_n}\frac{H_k(Y(t))}{\sqrt{k!}}\nu_d (dt)\right]^2 =  \E a_k^2(\xi)\int_{W_n}\int_{W_n} \frac{\E[ H_k(Y(t))H_k(Y(s))]}{k!}\nu_d(dt)\nu_d(ds)\\
    &= \E a_k^2(\xi) \int_{W_n}\int_{W_n}C^k(t-s)\nu_d(dt)\nu_d(ds)=\E a^2_k(\xi) \int_{\R^d}C^k(t)\nu_d(W_n \cap (W_n-t))\nu_d(dt).
\end{align*}
Introduce $\sigma^2_n:=\int_{\R^d}|C(t)|\nu_d(W_n \cap (W_n-t))\nu_d(dt).$
As before, 
\begin{align*}
    \E \left[\frac{1}{\nu_d(W_n)}\int_{W_n}F(Y(t),\xi)\nu_d(dt)-a_0(\xi)\right]^2 \leq \frac{\sigma^2_n}{\nu^2_d(W_n)} \sum_{k=1}^\infty \E a^2_k(\xi)\leq \frac{\sigma^2_n}{\nu^2_d(W_n)}\E\Psi\left(\frac{u}{\xi}\right)
\end{align*}
by Parseval identity, where 
 $a_0(\xi)=\int_{\R}\mathbbm{1}\{x \xi\geq u\}\vf(x)dx=\Psi\left({u}/{\xi}\right).$  So, if $\frac{\sigma^2_n}{\nu^2_d(W_n)}\to 0$ as  $n\to \infty,$ then 
{$\frac{1}{\nu_d(W_n)}\int_{W_n}\mathbbm{1}\{X(t)\geq u\}\nu_d(dt) \overset{d}{\longrightarrow} \Psi\left(\frac{u}{\xi}\right),\, n\to \infty.$}
\end{proof}
\begin{remark}{
A random field with random volatility is not ergodic. For a $\xi=c>0$ a.s., the convergence in Theorem \ref{RV:thm} is towards $\Psi(u/c).$ If $\xi$ is not a.s. constant, then the distribution function of random variable $\Psi\left({u}/{\xi}\right)$
 is  
 \begin{align*}
 &\pr\left(\Psi\left(\frac{u}{\xi}\right) \leq x \right)  = \pr\left({\xi}\Psi^{(-1)}\left(x\right) \leq  {u}\right)=\mathbbm{1}\{x= 1/2,u\geq 0\}\\
 &+\mathbbm{1}\{x< 1/2\}\pr\left(\xi \leq  \frac{u}{\Psi^{(-1)}(x) }\right)+\mathbbm{1}\{x> 1/2\}\pr\left(\xi \geq  \frac{u}{\Psi^{(-1)}(x) }\right),\,x\in [0,1].
 \end{align*}
}
{
For example, if $u>0$ and $\xi^2\sim \text{L\'{e}vy}(0,u^2)$ (cf. e.g. \cite{zolotarev}) then $\E \xi^2=+\infty$ and $\xi^2\stackrel{d}{=}u^2 (\Psi^{(-1)}(U))^{-2},$ where $U$ is a uniformly distributed random variable on [0,1]. Note that $\Psi(|\Psi^{(-1)}(x)|)=x\mathbbm{1}\{x\in[0,1/2)\}+(1-x)\mathbbm{1}\{x\in[1/2,1)\}=\min\{x,1-x\},$ $x\in[0,1].$ Thus, $\Psi\left(u/\xi\right) \stackrel{d}{=}\min\{U,1-U\},$ which is uniformly distributed on [0,1/2].}
\end{remark}

\begin{remark}
Condition \eqref{RV:eq1} holds, in particular, for {weakly dependent} random fields $Y,$ i.e., if $\int_{\R^d}|C(t)|\nu_d(dt)<\infty,$ since 
$$\frac{1}{\nu^2_d(W_n)}\int_{\R^d}|C(t)|\nu_d(W_n\cap(W_n-t))\nu_d(dt)\leq \frac{1}{\nu_d(W_n)}\int_{\R^d}|C(t)|\nu_d(dt) \xrightarrow[n\to\infty]{}0.$$
\end{remark}
\begin{remark}
Condition \eqref{RV:eq1} holds also for fields $Y$ such that $C(r_n y)\xrightarrow[n\to\infty]{}0$ uniformly in $y\in[-1,1]^d$ if $W_n=r_n V,$ $V\subseteq[-1,1]^d,$ $r_n\to+\infty$ as $n\to \infty,$  and $C$ is locally integrable on $\R^d.$ Indeed,
\begin{align*}
    &\frac{1}{\nu^2_d(W_n)}\int_{\R^d}|C(t)|\nu_d(W_n\cap(W_n-t))\nu_d(dt)\\
    &=\frac{1}{\nu^2_d(V)}\int_{[-r_n,r_n]^d}|C(t)|\nu_d\left(V\cap\left(V-\frac{t}{r_n}\right)\right)\nu_d\left(d\frac{t}{r_n}\right)\\
    &=|t=y r_n|=\frac{1}{\nu^2_d(V)}\int_{[-1,1]^d}|C(y r_n)|\nu_d(V\cap(V-y))\nu_d(y)\xrightarrow[n\to\infty]{}0.
\end{align*}
\end{remark}

\subsection{Fractional Gaussian processes and fields}
\begin{definition}
A centered Gaussian random field $\{G^H(t),t\in \R^d\}$ is called a {\it fractional Gaussian noise} with $H=(H_1,\ldots,H_d)$ if $\E [ G^H(t)G^H(s)]=C(t-s),$ $t,s\in \R^d$ with 
\begin{equation}
\label{fgn_def}
C(u)=\frac{1}{2^d}\prod_{i=1}^d\left(|u_i-1|^{2H_i}+|u_i+1|^{2H_i}-2|u_i|^{2H_i}\right),\,u=(u_1,\ldots,u_d)\in \R^d.
\end{equation}
\end{definition}

Evidently, the fractional Gaussian noise  is a stationary random field.  Moreover, a fractional Gaussian noise $\{G^H(t),t\in\R^d\}$  {is long-range dependent} if and only if $\max_{1\leq i\leq d}{H_i}>\frac{1}{2}.$ This follows directly from Definition \eqref{fgn_def} and the fact that function $|x-1|^{\alpha}+|x+1|^{\alpha}-2|x|^{\alpha},x\in \R$ is non-integrable for $\alpha>1.$

Now we prove a central limit theorem for the volumes of excursion sets of the fractional Gaussian noise. 
\begin{proposition}
\label{fgn_prop}
Let $\{G^H(t),t\in\R_+^d\}$ be a fractional Gaussian noise with index $H=(H_1,\ldots,H_d)\in (0,1)^d$ and there exists $m \in \{1,\ldots,d\}$ such that $H_m\in (1/2,1).$  Let $W_n=(a_n,b_n)\times U_n$ be a sequence of Borel subsets such that $r_n:=b_n-a_n \to +\infty, n\to\infty,$ and $0<c_1\leq \nu_{d-1}(U_n)\leq c_2<\infty.$ Then
\begin{equation}
\frac{\int_{W_n}\mathbbm{1}\{G^H(t)\geq u\}dt-\nu_d(W_n)\Psi(u)}{\vf(u)\sqrt{\int_{\R^d}C(t)\nu_d(W_n \cap (W_n-t) )dt}}\overset{d}{\longrightarrow} N(0,1),n\to\infty,
\end{equation}
where $C(\cdot)$ is the covariance function \eqref{fgn_def}.
\end{proposition}

\begin{proof}
For $\alpha\in (0,1),$ consider the function
\begin{equation}
\label{rho-def} 
\rho_\alpha(s):=|s+1|^{2\alpha}+ |s-1|^{2\alpha}-2 |s|^{2\alpha},
\end{equation}
Prove that $\rho_\alpha(s)>0,s\geq 0$ if $\alpha>1/2.$

Consider the case $s\in(0,1).$ Then we have $\rho_\alpha(s)=(1-s)^{2\alpha}+(s+1)^{2\alpha}-2s^{2\alpha}.$ Using binomial series representation $(1+x)^{\alpha}=\sum_{k=0}^\infty \binom{\alpha}{k}x^k,|x|\leq 1,$ where $\binom{\alpha}{k}:=\frac{\alpha(\alpha-1)\cdots(\alpha-k+1)}{k!}$, we get
\begin{equation*}
    \rho_\alpha(s)=1+\sum_{k=1}^\infty \binom{2\alpha}{k}(-1)^k s^k+1+\sum_{k=1}^\infty \binom{2\alpha}{k}s^k -2 s^{2\alpha}=2(1-s^{2\alpha})+\sum_{k=1}^\infty \binom{2\alpha}{2k}s^{2k}.
\end{equation*}
It is easy to check that $\binom{2\alpha}{2k}\geq 0,k\in \N$ if $\alpha\in (1/2,1).$ So,
$\rho_\alpha(s)\geq 2(1-s^{2\alpha})>0,$ $s\in (0,1).$

Consider the case $s\geq 1.$ Then we have $\rho_\alpha(s)=s^{2\alpha}\left(\left(1-\frac{1}{s}\right)^{2\alpha}+\left(1+\frac{1}{s}\right)^{2\alpha}-2\right).$ Use the binomial series representation:
{\begin{align}
    \frac{\rho_\alpha(s)}{s^{2\alpha}}&=1+\sum_{k=1}^\infty \binom{2\alpha}{k}(-1)^k \frac{1}{s^k}+1+\sum_{k=1}^\infty \binom{2\alpha}{k} \frac{1}{s^k}-2
    \label{rhoser}=\sum_{k=1}^\infty \binom{2\alpha}{2k} \frac{1}{s^{2k}}> 0.
\end{align}}

Then we estimate $\int_0^r\rho_\alpha(s)ds,$ $r>1:$
\begin{align*}
\int_0^r \rho_\alpha(s)ds&\geq \int_1^r \rho_\alpha(s)ds=2\int_1^r s^{2\alpha}\sum_{k=1}^\infty \binom{2\alpha}{2k} \frac{1}{s^{2k}}ds\\
&\geq 2\int_1^r s^{2\alpha} \binom{2\alpha}{2} \frac{1}{s^{2}}ds=2\alpha (2\alpha-1)\int_1^r s^{2\alpha-2}ds=2\alpha(r^{2\alpha-1}-1).
\end{align*}
Similarly to the proof of Corollary \ref{cor1}, we put $\delta=(1-H_m)/2$ and use \eqref{ineq11} to get 
\begin{align*}
&\left|\frac{\int_{U_n}\int_{U_n}\prod _{1\leq i\leq d,i\neq m}\rho^2_{H_i}(t_i-s_i)dt d s \int_{a_n}^{b_n}\int_{a_n}^{b_n}\rho^2_{H_m}(t_m-s_m)dt_m  ds_m}{\int_{U_n}\int_{U_n}\prod_{1\leq i\leq d,i\neq m}\rho_{H_i}(t_i-s_i)dt d s \int_{a_n}^{b_n}\int_{a_n}^{b_n}\rho_{H_m}(t_m-s_m)dt_m  ds_m}\right|\\
&\leq \frac{\sup_{n\geq 1}\int_{U_n}\int_{U_n}\prod_{1\leq i\leq d,i\neq m}\rho^2_{H_i}(t_i-s_i)dt d s}{\left|\inf_{n\geq 1}\int_{U_n}\int_{U_n}\prod_{1\leq i\leq d,i\neq m}\rho_{H_i}(t_i-s_i)dt d s\right|}\frac{ \int_{0}^{r_n}\rho_{H_m}^2(v)(r_n-v)dv}{\int_{0}^{r_n}\rho_{H_m}(v)(r_n-v)dv}\\
&\leq const \left(\frac{2^{1+\delta}(r_n/2)^{\delta}}{\int_{0}^{r_n/2}\rho_{H_m}(v)dv}+\sup_{v\geq r_n^{\delta}}\rho_{H_m}(v)\right)\longrightarrow 0,\quad n\to \infty.
\end{align*}
%where $L_H>0$ is some constant dependent on $(H_1,\ldots,H_d).$
application of Corollary \ref{ch1:cor5} finishes the proof.
\end{proof}

%In case when for some $i\in \overline{1,d}$: $H_i\in (0,1/2)$ the fractional Gaussian noise $G^H$ is not \textbf{PA}. In this case, the limiting behaviour of its extrusions depends on the growth rate of $W_n.$  
\begin{proposition}
\label{fgn_cor}
Let $\{G^H(t),t\in\R^d\}$ be a fractional Gaussian noise with index $H=(H_1,\ldots,H_d)\in (0,1)^d.$  Let $W_n=\prod_{i=1}^d [0,r_{n,i}]$ be such that $r_{n,i} \to +\infty, n\to\infty$  for all $1\leq i\leq d$ 
and 
\begin{equation}
\label{rncond}
\prod_{i=1}^d r_{n,i}^{\delta_i +1- 2H_i} \to 0, \quad n\to\infty,
\end{equation}
where $\delta_i=\frac{2H_i-1}{3-2H_i}\mathbbm{1}\{2H_i>1\},$ $1\leq i\leq d.$
Then
\begin{equation}
\label{cor_eq}
\frac{\int_{W_n}\mathbbm{1}\{G^H(t)\geq u\}dt-\prod_{i=1}^d{r_{n,i}}\Psi(u)}{\vf(u)\prod_{i=1}^dr^{H_i}_{i,n}}\xrightarrow[n\to\infty]{d} N(0,1).
\end{equation}
\end{proposition}
\begin{proof}
Apply Corollary \ref{ch1:cor5} to $X=G^H$ and $f(x)=x.$ We have
%. From the proof of Proposition \ref{fgn_prop} we know that 
\begin{align*}
    &\int_{\R^d}C(t)\nu_d(W_n \cap (W_n-t))\nu_d(dt)=
    \int_{\R^d}C(t)\prod_{i=1}^d \max\left(r_{n,i}-|t_i|,0\right)\nu_d(dt)\\
    &=\prod_{i=1}^d r_{n,i}\int_{-r_{n,i}}^{r_{n,i}} (|t_i+1|^{2H_i}+|t_i-1|^{2H_i}-2 |t_i|^{2H_i})\left(1-\frac{t_i}{r_{n,i}}\right)dt_i\\
    &=\prod_{i=1}^d r_{n,i} \left(\int_{-r_{n,i}}^{r_{n,i}}\rho_{H_i}(t_i)d t_i-\frac{1}{r_{n,i}}\int_{-r_{n,i}}^{r_{n,i}}\rho_{H_i}(t_i)|t_i|d t_i\right),
\end{align*}
where $\rho_\alpha$ is defined in \eqref{rho-def}.

By direct calculation, we get for $r>1$
\begin{align}
\nonumber\int_{-r}^{r}\rho_{\alpha}(v)d v&=2\left(\int_{0}^{r}((v+1)^{2\alpha}-2 v^{2\alpha})dv+\int_{0}^{1}(1-v)^{2\alpha}dv+\int_{1}^{r}(v-1)^{2\alpha}dv\right)\\
\label{eq32}&=\frac{2}{2\alpha+1}\left((r+1)^{2\alpha+1}+(r-1)^{2\alpha+1}-2r^{2\alpha+1}\right),
\end{align}
and
\begin{align}
\nonumber&\int_{-r}^{r}\rho_{\alpha}(v) |v| d v=2\int_{0}^{r}(v(v+1)^{2\alpha}-2 v^{2\alpha+1})dv+2\int_{0}^{1}v(1-v)^{2\alpha}dv\\
\nonumber&+2\int_{1}^{r}v(v-1)^{2\alpha}dv
%\nonumber&=\frac{1}{2\alpha+1}\left((r+1)^{2\alpha+1}\left(r-\frac{r+1}{2\alpha+2}\right)+(r-1)^{2\alpha+1}\left(r-\frac{r-1}{2\alpha+2}\right)\right)\\
%&-\frac{1}{\alpha+1}r^{2\alpha+2}+\frac{1}{(2\alpha+1)(\alpha+1)}.
=\frac{(r+1)^{2\alpha+1}((2\alpha+1)r-1)+(r-1)^{2\alpha+1}((2\alpha+1)r+1)}{(\alpha+1)(2\alpha+1)}\\
\label{eq33}&+\frac{2-2(2\alpha+1)r^{2\alpha+2}}{(\alpha+1)(2\alpha+1)}.
\end{align}
Therefore, combining \eqref{eq32}, \eqref{eq33} and series representation \eqref{rhoser}, we get
\begin{align}
\nonumber\int_{-r}^{r}\rho_{\alpha}(v)\left(1-\frac{|v|}{r}\right) d v&=\frac{1}{r(2\alpha+1)(\alpha+1)}\left((r+1)^{2\alpha+2}+(r-1)^{2\alpha+2}-2r^{2\alpha+2}-2\right)\\
\nonumber&=\frac{2r^{2\alpha+1}}{(2\alpha+1)(\alpha+1)}\left(\frac{(2\alpha+1)(2\alpha +2)}{2 r^2}+\sum_{k=2}^{\infty}\binom{2\alpha+2}{2k}\frac{1}{r^{2k}}\right)\\
\label{eq34}&-\frac{2}{(2\alpha+1)(\alpha+1)r}\sim 2r^{2\alpha-1},\quad r\to\infty.
\end{align}
Thus, we obtain {$
 \int_{W_n}C(t)\nu_d(W_n \cap (W_n-t))\nu_d(dt)
    {\sim} 2^d\prod_{i=1}^d r_{n,i}^{2H_i}$ as $r_{n,i}\to\infty.$}
    
We check condition \eqref{condcor1}. Note that  $\int_\R \rho^2_{H_i}(v)dv<\infty$ and $ \int_\R \rho^2_{H_i}(v)vdv<\infty$ if $2H_i<1,$ therefore   
$$\frac{\int_{-r_{n,i}}^{r_{n,i}} \rho^2_{H_i}(v)(r_{n,i}-v)dv}{\int_{-r_{n,i}}^{r_{n,i}} \rho_{H_i}(v)(r_{n,i}-v)dv}\sim r_{n,i}^{1-2H_i},\quad n\to \infty.$$

If $2 H_i>1,$ then  $\rho_{H_i}(v)>0,v>0$ and $\rho_{H_i}$ is non-increasing {when} $v>1$. We can use the same arguments as in Proposition \ref{fgn_prop} and get 
\begin{align*}
\frac{\int_{-r_{n,i}}^{r_{n,i}} \rho^2_{H_i}(v)(r_{n,i}-|v|)dv}{\int_{-r_{n,i}}^{r_{n,i}} \rho_{H_i}(v)(r_{n,i}-|v|)dv}&\leq \frac{2^{1+\delta_i}(r_{n,i}/2)^{\delta_i}}{\int_{0}^{r_{n,i}/2}\rho_{H_i}(v)dv}+\sup_{v\geq r_{n,i}^\delta}\rho_{H_i}(v)\\
&\sim const \left(r_{n,i}^{\delta_i+1-2H_i} +r_{n,i}^{2\delta_i(1-H_i)}\right),\quad n\to \infty.
\end{align*}
Choosing $\delta_i=\frac{2H_i-1}{3-2H_i}\mathbbm{1}\{2H_i>1\},$ $1\leq i\leq d,$ we get that 
$$\lim_{n\to \infty} \frac{\int_{W_n}\int_{W_n} C^2(t-s)dtds}{\int_{W_n}\int_{W_n} C(t-s)dtds} \leq const \lim_{n\to \infty} \prod_{i=1}^d r_{n,i}^{\delta_i +1- 2H_i}, $$ 
which ends the proof of the Proposition.
\end{proof}

\begin{remark}
\label{rem:46}
Let $r_{n,i}=r_n^{\gamma_i},$ with $\gamma_i=(3-2 H_i)(3-2H_i-\mathbbm{1}\{2H_i>1\})^{-1},$ $1\leq i\leq d,$ and $r_n\to +\infty,n\to\infty.$ Then condition \eqref{rncond} is fulfilled  if $\sum_{i=1}^d H_i >d/2$ and \eqref{cor_eq} rewrites
\begin{equation}
\label{rem_eq}
\frac{\int_{W_n}\mathbbm{1}\{G^H(t)\geq u\}dt-r_n^{\sum_{i=1}^d\gamma_i}\Psi(u)}{\vf(u) r_n^{\sum_{i=1}^d H_i \gamma_i}}\xrightarrow[n\to\infty]{d} N(0,1).
\end{equation}
\end{remark}
Let us consider the case $d=2,$ $r_n=n$ and $H_1<1/2,H_2>1/2,$ such that $H_1+H_2>1.$ Then we get from Remark \ref{rem:46} that $\gamma_1=1,$ $\gamma_2=\frac{3-2H_2}{2-2H_2}.$ 
It is interesting to compare our result \eqref{rem_eq}, which now reads as 
$$\frac{\int_{[0,n]\times [0,n^{\gamma_2}]}\mathbbm{1}\{G^{H_1,H_2}(t_1,t_2)\geq u\}dt_1dt_2-n^{1+\gamma_2}\Psi(u)}{\vf(u) n^{H_1+H_2\gamma_2}}\xrightarrow[n\to\infty]{d} N(0,1),$$ with the results of paper \cite{surg}. Note that 
the spectral density $f$ of $G^H$ is proportional to 
$|x_1|^{1-2H_1} |x_2|^{1-2H_2}$ as $x_1,x_2\to 0$ (see e.g. \cite{Ma1968}). Then application of \cite[Proposition 3.2.]{surg} to partial sums of $G^{H_1,H_2}$ gives 
$$\frac{\kappa(H_1)\kappa(H_2)}{n^{H_1+H_2\gamma_2}}\sum_{1\leq k_1 \leq n,1\leq k_2 \leq n^\gamma_2}G^{H_1,H_2}(k_1,k_2) \xrightarrow[n\to\infty]{d} N(0,1),$$
where $\kappa_1,\kappa_2$ are normalizing constants.

%Consider a fractional Brownian motion $\{B^H(t),t\in \R_+\}.$ Let us try to apply Theorem \ref{mainthm} to a normalised fractional Brownian motion $\bar{B}^H(t)=t^{-H}B^H(t),t>0.$
%In this case $\rho(t,s)=\Corr(\bar{B}^H(t),\bar{B}^H(s))=t^{-H}s^{-H}\Cov(B^H(t),B^H(s)=\frac{1}{2}t^{-H}s^{-H}(t^{2H}+s^{2H}-|t-s|^{2H}).$
%Function $\rho$ is invariant, that is, $\forall a>0:$ $\rho(at,as)=\frac{1}{2}(at)^{-H}(as)^{-H}((at)^{2H}+(as)^{2H}-|at-as|^{2H})=\rho(t,a).$

\section*{Funding}
The first author was supported by DFG Grant 390879134.
\bibliographystyle{abbrv}

\bibliography{lit2}

\end{document}